\documentclass[leqno, 10pt, a4paper]{amsart}
\usepackage{amsmath}
\usepackage{amssymb, latexsym, mathrsfs, a4wide, mathabx, dsfont}
\usepackage{xcolor}

 \newtheorem{theorem}{Theorem}[section]
 \newtheorem{lemma}[theorem]{Lemma}
 \newtheorem{proposition}[theorem]{Proposition}
 \newtheorem{corollary}[theorem]{Corollary}

 \newtheorem*{theorem*}{Theorem}
\newtheorem*{proposition*}{Proposition}
\newtheorem*{lemma*}{Lemma}

\theoremstyle{definition}
 \newtheorem{definition}[theorem]{Definition}

 \theoremstyle{remark}
 \newtheorem{example}[theorem]{Example}
 \newtheorem{remark}[theorem]{Remark}

   \newtheorem*{claim*}{Claim}

\newcommand{\op}[1]{\operatorname{#1}}

\newcommand{\acou}[2]{\ensuremath{\left\langle #1 , #2 \right\rangle}} 
\newcommand{\acout}[2]{\ensuremath{\langle #1 , #2\rangle}} 
 
\newcommand{\scal}[2]{\ensuremath{\left\langle #1 | #2 \right\rangle}} 
\newcommand{\scalt}[2]{\ensuremath{\langle #1 | #2 \rangle}} 

\newcommand{\acoup}[2]{\ensuremath{\left(#1,#2\right)}}

\newcommand{\Tr}{\ensuremath{\op{Tr}}}

\def\XXint#1#2#3{{\setbox0=\hbox{$#1{#2#3}{\int}$}
\vcenter{\hbox{$#2#3$}}\kern-.5\wd0}}

\newcommand{\GL}{\op{GL}}

\newcommand{\C}{\ensuremath{\mathbb{C}}} 
 
\newcommand{\N}{\ensuremath{\mathbb{N}}} 
 
\newcommand{\R}{\ensuremath{\mathbb{R}}} 
 
\newcommand{\T}{\ensuremath{\mathbb{T}}} 
\newcommand{\Z}{\ensuremath{\mathbb{Z}}}

\newcommand{\fS}{\ensuremath{\mathfrak{S}}}

\newcommand{\Ca}[1]{\ensuremath{\mathcal{#1}}}
\newcommand{\cA}{\Ca{A}}
\newcommand{\cB}{\Ca{B}}

\newcommand{\cE}{\Ca{E}}

\newcommand{\cH}{\ensuremath{\mathcal{H}}}

\newcommand{\cL}{\ensuremath{\mathcal{L}}}

\newcommand{\cX}{\mathscr{X}}

\newcommand{\Hom}{\op{Hom}}

\newcommand{\dom}{\op{dom}}

\newcommand{\Sp}{\op{Sp}}
\newcommand{\Vol}{\op{Vol}}

\newcommand{\opp}{\textup{o}}

\newcommand{\dl}{\partial}
\newcommand{\dive}{\op{div}}

\numberwithin{equation}{section}

\begin{document}

\title{Laplace-Beltrami Operators on Noncommutative Tori}

\author{Hyunsu Ha}
 \address{Department of Mathematical Sciences, Seoul National University, Seoul, South Korea}
 \email{mamaps@snu.ac.kr}

\author{Rapha\"el Ponge}
 \address{School of Mathematics, Sichuan University, Chengdu, China}
 \email{ponge.math@icloud.com}

 \thanks{The research for this article was partially supported by 
  NRF grants 2013R1A1A2008802 and 2016R1D1A1B01015971 (South Korea).}
 
\begin{abstract}
 In this paper, we construct Laplace-Beltrami operators associated with arbitrary Riemannian metrics on noncommutative tori of any dimension. These operators enjoy the main properties of the Laplace-Beltrami operators on ordinary Riemannian manifolds. The construction takes into account the non-triviality of the group of modular automorphisms. On the way we introduce notions of Riemannian density and Riemannian volumes for noncommutative tori. \end{abstract}

\maketitle 

\section{Introduction}
Noncommutative tori are among the most well known examples of noncommutative spaces. For instance, noncommutative 2-tori arise from actions on circles by irrational rotations. Following seminal work of Connes-Tretkoff~\cite{CT:Baltimore11} and Connes-Moscovici~\cite{CM:JAMS14}, a very active current trend in noncommutative geometry is the building of a differential geometric apparatus on noncommutative tori that take into account 
the non-triviality of their modular automorphism groups (see, e.g., \cite{CF:MJM19, CM:JAMS14, CT:Baltimore11, DS:SIGMA15, DGK:arXiv18, Fa:JMP15, FK:JNCG12, FK:LMP13, FK:JNCG13, FK:JNCG15, FGK:JNCG19, LM:GAFA16, Liu:arXiv18a, Liu:arXiv18b, Ro:SIGMA13}). So far the main focus has been mostly on conformal deformations of the flat Euclidean metric or products of such metrics. There is a general notion of general metric on noncommutative tori due to Rosenberg~\cite{Ro:SIGMA13}. In particular, the analogue of Levi-Civita theorem holds (see~\cite{Ro:SIGMA13}). Numerous important results in Riemannian geometry and spectral geometry arise from the analysis of the Laplace-Beltrami operator acting on functions on a closed Riemannian manifolds. Therefore, it seems only timely to construct Laplace-Beltrami operators associated with arbitrary Riemannian metrics. This is precisely the main goal of this article. 

We refer to Section~\ref{sec:Laplace-Beltrami} for the full details of the construction of the Laplace-Beltrami operators on noncommutative tori. The construction is actually carried out for arbitrary Hermitian metrics and positive densities on noncommutative tori, which is a bit more general. Anyway, the point is that the Laplace-Beltrami operators on noncommutative tori satisfy the main classical properties of Laplace-Beltrami operators on closed manifolds. In particular, they are selfadjoint elliptic differential operators (see Proposition~\ref{prop:Laplacian.positivity-ellipticity} and Proposition~\ref{prop:Laplacian.spectral-theory}). As a result, they enjoy the usual elliptic regularity properties (Proposition~\ref{prop:Laplacian.elliptic-regularity}). At the spectral level, the spectra are unbounded discrete sets that consist of eigenvalues with finite multiplicities (see Proposition~\ref{prop:Laplacian.spectral-theory}) and we have orthonormal bases of smooth eigenvectors (Proposition~\ref{prop:Laplacian.eigenbasis}). We also describe the transformation of Laplace-Beltrami operators under conformal changes of metrics (see Proposition~\ref{prop:Laplacian.conformal-transformation}). In particular, in dimension~2 we obtain a conformal covariance which is the very analogue of the well known conformal covariance of the Laplace-Beltrami operator on Riemannian surfaces.  

Denoting by $\cA_\theta$ the (smooth) noncommutative torus associated with a skew-symmetric $n\times n$ matrix $\theta$, the corresponding space of vector fields $\cX_\theta$ is the left $\cA_\theta$-module generated by the canonical derivations $\dl_1, \ldots, \dl_n$ of $\cA_\theta$. The Hermitian metrics on $\cX_\theta$ are in one-to-one correspondence with positive invertible element matrices  $h=(h_{ij})$ in $M_n(\cA_\theta)$. A Riemannian metric is given by such a matrix $g=(g_{ij})$ with the further requirement that its entries and those of its inverse $g^{-1}=(g^{ij})$ are selfadjoint.  A class of Riemann metrics of special interest consists of the self-compatible metrics $g=(g_{ij})$ for which the entries $g_{ij}$ mutually commute with each other. This includes the conformal deformations of the Euclidean flat metric considered by many authors and the functional metrics of~\cite{GK:arXiv18}. We refer to Section~\ref{sec:Riem} for a more detailed account on Riemannian metrics on noncommutative tori. 

As mentioned above the Laplace-Beltrami operators on noncommutative tori enjoy the main properties of their counter-parts on ordinary manifolds. However, their constructions bear a few differences due to the noncommutativity of noncommutative tori. The main influx of this noncommutativity concerns the construction of the relevant inner products involved in the definition of the Laplace-Beltrami operators. In the commutative case the relevant inner products are constructed out of Hermitian metrics and positive densities. The latter are given by integration against positive functions. Analogously, a positive density on $\cA_\theta$ is given by a positive invertible element $\nu\in \cA_\theta$ and its associated weight,   
\begin{equation}
 \varphi_{\nu}(u)= (2\pi)^{n}\tau[ u\nu], \qquad u\in \cA_\theta, 
 \label{eq:Intro.weight}
\end{equation}
where $\tau$ is the (standard) normalized trace of $\cA_\theta$. This allows us to carry out the GNS construction and get a $*$-representation of $\cA_\theta$ in the Hilbert space $\cH_\nu$ that arises from the completion of $\cA_\theta$ with respect to the inner product, 
\begin{equation*}
 \scal{u}{v}_\nu := (2\pi)^{-n} \varphi_\nu(v^*u)= \tau (u\nu v^*), \qquad u,v\in \cA_\theta. 
\end{equation*}

Due to the noncommutativity of $\cA_\theta$, the weight $\varphi_\nu$ is not a trace (unless $\nu$ is a scalar). As a result the right multiplication of $\cA_\theta$ does not provide us with a $*$-representation on $\cH_\nu$ of the opposite algebra $\cA_\theta^\opp$, i.e., we have a non-trivial group of modular automorphisms. As in~\cite{CM:JAMS14, CT:Baltimore11} this is fixed by using the inner automorphism,
\begin{equation*}
 \sigma_\nu(u) = \nu^{\frac12} u \nu^{-\frac12}, \qquad u\in \cA_\theta. 
\end{equation*}
In other words, we obtain a $*$-representation of $\cA_\theta^\opp$ in the Hilbert space $\cH_\nu^\opp$ given by the completion with respect to the inner product, 
\begin{equation}
 \scal{u}{v}_\nu^\opp:= \varphi_\nu(uv^*)= \scal{\sigma_\nu(v)}{\sigma_\nu(v)}_\nu, \qquad u,v\in \cA_\theta.
 \label{eq:Intro.inner-product} 
\end{equation}
In particular, $\mathbf{\Delta}(u)=\sigma_\nu^2(u)=\nu u \nu^{-1}$ and $J(u)=\sigma_\nu(u^*)$ are the modular operator and anti-linear involution provided by Tomita-Takesaki theory.  In some sense, the bulk of the construction of the Laplace-Beltrami operators is the accounting of the modular automorphism group. In fact, this accounting is necessary in order to obtain Laplace-Beltrami operators that are differential operators.

The space of differential forms $\Omega_\theta^1$ on $\cA_\theta$ is defined as the dual module of the module of vector fields $\cX_\theta$. As $\cX_\theta$ is a left $\cA_\theta$-module, we obtain a right $\cA_\theta$-module and the canonical derivations $\dl_1,\ldots, \dl_n$ give rise to a dual basis $\theta^1, \ldots, \theta^n$ of $\Omega_\theta^1$. The differential $d:\cA_\theta \rightarrow \Omega_\theta^1$ is then given by
\begin{equation*}
 du = \theta^1 \dl_1(u)+\cdots + \theta^n \dl_n(u), \qquad u\in \cA_\theta. 
\end{equation*}
By duality any Hermitian metric $h=(h_{ij})$ gives rise a Hermitian metric $(h^{-1})^t=(h^{ji})$ on $\Omega_\theta^1$. In addition, the basis $\theta^1,\ldots, \theta^n$ gives rise to a left $\cA_\theta$-module structure on $\Omega^1_\theta$ (see~Section~\ref{sec:1-forms}). Given any positive density $\nu\in\cA_\theta$, this allows us to lift the modular inner automorphism $\sigma_\nu$ to $\Omega_\theta^1$. Much in the same way as in~(\ref{eq:Intro.inner-product}), we then can equip $\Omega_\theta^1$ with the inner product, 
\begin{equation*}
 \scal{\omega}{\zeta}_{h,\nu}^\opp = \varphi_\nu\left[\acoup{\sigma_\nu(\omega)}{\sigma_\nu(\zeta)}\right], \qquad \omega, \zeta\in \Omega^1_\theta. 
\end{equation*}

There is a natural notion of divergence of vector fields in $\cX_\theta$ associated with the positive density $\nu$ (see Section~\ref{sec:1-forms}). By duality this gives rise to a divergence operator $\delta:\Omega_\theta^1 \rightarrow \cA_\theta$. As usual, the divergence operator yields the formal adjoint of the differential $d$ on $\Omega^1_\theta$; in our setting this is with respect to the inner products $\scal{\cdot}{\cdot}_\nu^\opp$ and $\scal{\cdot}{\cdot}_{h,\nu}^\opp$ described above (see Proposition~\ref{prop:1-forms.divergence}). The Laplace-Beltrami operator $\Delta_{h,\nu}:\cA_\theta \rightarrow \cA_\theta$ is then defined by
\begin{equation}
 \Delta_{h,\nu} u = -\delta (du), \qquad u \in \cA_\theta. 
 \label{eq:Intro.Laplace-Beltrami1}
\end{equation}
This is the usual definition of the Laplace-Beltrami operator (up to the sign convention). It can be further shown (\emph{cf}.\ Proposition~\ref{prop:Laplacian.positivity-ellipticity}) that we have
\begin{equation}
  \Delta_{h,\nu} u = - \nu^{-1} \sum_{1\leq i,j \leq n} \dl_i \big( \sqrt{\nu} h^{ij} \sqrt{\nu} \dl_j(u)\big), \qquad u\in \cA_\theta. 
   \label{eq:Intro.Laplace-Beltrami2}
\end{equation}
This is the usual formula for the Laplace-Beltrami operator, except for the replacement of the term $ h^{ij}\nu$ by $\sqrt{\nu} h^{ij} \sqrt{\nu}$, which accounts for the noncommutativity of $\cA_\theta$.

In order to define the Laplace-Beltrami operator associated with a Riemannian metric we only need to define the analogue of the Riemannian density. 
If $g=(g_{ij}(x))$ is a Riemannian metric on an ordinary torus, or more generally on any ordinary manifold, the Riemannian density is given $\nu_g(x)=\sqrt{\det(g(x)}$. Thus, in order to define the analogue of the Riemannian density for $\cA_\theta$ we need a notion of determinant for positive invertible elements of $M_n(\cA_\theta)$. 
As $\cA_\theta$ is not a commutative ring, we cannot define the determinant in the usual way. Following the approach of Fuglede-Kadison~\cite{FK:AM52}, we define the determinant of a positive invertible matrix $h=(h_{ij})$ in $M_n(\cA_\theta)$ by
\begin{equation*}
 \det(h) = \exp\left( \Tr \big[ \log h\big] \right),
\end{equation*}
where $\log h \in M_n(\cA_\theta)$ is defined by holomorphic functional calculus and $\Tr: M_n(\cA_\theta) \rightarrow \cA_\theta$ is the usual  matrix ``trace'' on $M_n(\cA_\theta)$. This defines a positive invertible element of $\cA_\theta$. When $h$ is self-compatible we recover the usual determinant (see Corollary~\ref{cor:det.self-comp}). We refer to Section~\ref{sec:determinant}, for a description of the main properties of this determinant. We stress that, due to the noncommutativity of $\cA_\theta$, the matrix trace $\Tr$ is not actually a trace, and so this determinant cannot be multiplicative (compare~\cite{FK:AM52}). However, we still have some instances of multiplicativity (see Proposition~\ref{prop:det.product} and Proposition~\ref{prop:det.conjugaison}). 

Given any Riemannian metric $g=(g_{ij})$ on $\cA_\theta$, its determinant $\det(g)$ is a positive element of $\cA_\theta$. This allows us to define the \emph{Riemannian density} by  
\begin{equation*}
 \nu(g):= \sqrt{\det(g)}. 
\end{equation*}
As in~(\ref{eq:Intro.weight}) this gives rise to a weight $\varphi_g:=\varphi_{\nu(g)}$, called Riemannian weight. The \emph{Riemannian volume} then is
\begin{equation*}
 \Vol_g(\cA_\theta):=\varphi_g(1)= (2\pi)^n\tau\left[ \nu(g)\right]. 
\end{equation*}
We refer to Section~\ref{sec:Riem-volume} for the main properties of Riemannian densities and Riemannian volumes and the computations of a few examples.  

Given a Riemannian metric $g=(g_{ij})$ on $\cA_\theta$, its Laplace-Beltrami operator $\Delta_g:\cA_\theta \rightarrow \cA_\theta$ is the Laplace-Beltrami operator $\Delta_{h,\nu}$ given by~(\ref{eq:Intro.Laplace-Beltrami1})--(\ref{eq:Intro.Laplace-Beltrami2}) with $h^{ij}=g^{ij}$ and $\nu=\nu(g)$, where $g^{ij}$ are the entries of $g^{-1}$. In particular, when $g$ is a self-compatible metric, it can be shown that we have
\begin{equation*}
\Delta_gu= - \nu(g)^{-1} \sum_{1\leq i,j \leq n} \dl_i \big(  g^{ij} \nu(g) \dl_j(u)\big), \qquad u\in \cA_\theta.
\end{equation*}
This is the full analogue of the usual formula for the Laplace-Beltrami operator on a Riemannian manifold. When $n=2$ and $g$ is a conformal deformation of the Euclidean flat metric we recover the conformally deformed Laplacian of Connes-Tretkoff~\cite{CT:Baltimore11} up to unitary equivalence (see Example~\ref{ex:Laplacian.conformal-flat}). 

As mentioned above, the Laplace-Beltrami operators on noncommutative tori enjoy the main properties of Laplace-Beltrami operators on ordinary Riemannian manifolds. Therefore, it stands for reason to expect that many of well-known spectral geometry results related to Laplace-Beltrami operators on Riemannian manifolds should have analogues for noncommutative tori. Instances of such results are analogues for noncommutative tori of Weyl' law and its local and microlocal versions (see~\cite{Po:Weyl}  and Remark~\ref{rmk:Laplacian.Weyl}). In addition, it would be important to extend the definition of the differential $d$ and the Laplace-Beltrami operator to operators on differential forms of any degree. That is, we seek for a de Rham complex and Hodge Laplacians on noncommutative tori so as to have a full Hodge theory on noncommutative tori. Note that, here again, the noncommutativity of noncommutative tori prevent us form defining higher degree differential forms in the usual way. 

This paper is organized as follows. In Section~\ref{sec:NCtori}, we review the main definitions regarding noncommutative tori. In Section~\ref{sec:positivity}, 
we review some facts on positive elements and Hermitian modules in the setting of smooth $*$-algebras. In Section~\ref{sec:Riem}, after setting up the definition of a Riemannian metric for noncommutative tori, we exhibit a few examples and introduce a notion of conformal equivalence for such metrics.  
In Section~\ref{sec:densities}, after discussing the noncommutative analogues of 
positive densities on noncommutative tori, we explain how this enables us to define Hermitian inner products on Hermitian modules over noncommutative tori. In Section~\ref{sec:determinant},
we define the determinant of positive invertible matrices in $M_m(\cA_\theta)$, establish some of its properties, and compute some examples. 
In Section~\ref{sec:Riem-volume}, we introduce Riemannian densities and Riemann volumes for arbitrary Riemannian metrics on noncommutative toris and check their main properties. 
In Section~\ref{sec:1-forms}, we describe the bimodule of differential 1-forms on noncommutative toris and introduce a divergence operator on this bimodule. 
In Section~\ref{sec:Laplace-Beltrami}, we define and study the main properties of Laplace-Beltrami operators associated with 
arbitrary Hermitian metrics and positive densities on noncommutative tori. 

\subsection*{Acknowlegements} The authors would like to thank Yang Liu, Ed McDonald, Fedor Sukochev, and Dmitriy Zanin for discussions related to the subject matter of this paper. R.P.~also thanks University of New South Wales (Sydney, Australia) and University of Qu\'ebec at Montr\'eal (Canada)  for their hospitality during the preparation of this manuscript.

\section{Noncommutative Tori}\label{sec:NCtori}
In this section, we recall the main definitions and properties of noncommutative tori. We refer to~\cite{Co:NCG, HLP:Part1, Ri:PJM81, Ri:CM90} and the references therein for a more comprehensive account.  

Throughout this chapter, we let $\theta =(\theta_{jk})$ be a real anti-symmetric $n\times n$-matrix ($n\geq 2$). Let $\T^n=\R^n\slash (2\pi \Z)^n$ be the ordinary $n$-torus and $L^2(\T^n)$ the Hilbert space of $L^2$-functions on $\T^n$ equipped with the inner product, 
\begin{equation} 
 \scal{\xi}{\eta}= (2\pi)^{-n}\int_{\T^n}\xi(x) \overline{\eta(x)} d x, \qquad \xi, \eta \in L^2(\T^n). 
 \label{eq:NCtori.innerproduct-L2} 
\end{equation}
 For $j=1,\ldots, n$ let $U_j:L^2(\T^n)\rightarrow L^2(\T^n)$ be the unitary operator given by 
 \begin{equation*}
 \left( U_j\xi\right)(x)= e^{ix_j} \xi\left( x+\pi \theta_j\right), \qquad \xi \in L^2(\T^n), 
\end{equation*}
where $\theta_j$ is the $j$-th column vector of $\theta$. We then have the relations, 
 \begin{equation*}
 U_kU_j = e^{2i\pi \theta_{jk}} U_jU_k, \qquad j,k=1, \ldots, n. 
\end{equation*}
The \emph{noncommutative torus} $A_\theta$  is the $C^*$-algebra generated by the unitaries $U_1, \ldots, U_n$. 

A dense subspace of $A_\theta$ is the span $\cA_\theta^0$ of the unitaries, 
\begin{equation*}
 U^k:=U_1^{k_1} \cdots U_{n}^{k_n}, \qquad k=(k_1, \ldots, k_n)\in \Z^n. 
\end{equation*}
When $\theta=0$ each unitary $U_j$ is simply the multiplication by $e^{ix_j}$. Thus, in this case $\cA_\theta^0$ is the space of trigonometric polynomials and $A_\theta$ is precisely the $C^*$-algebra of continuous functions on $\T^n$. 

The action of $\R^n$ on $\T^n$ by translation yields a unitary representation $s\rightarrow V_s$ of $\R^n$ given by
\begin{equation*}
\left( V_s \xi\right)(x)=\xi(x+s), \qquad \xi \in L^2(\T^n), \ s\in \R^n. 
\end{equation*}
We then get an action $(s,T)\rightarrow \alpha_s(T)$ of $\R^n$ on $\cL(L^2(\T^n))$ given by 
\begin{equation*} 
 \alpha_s(T)=V_s TV_s^{-1}, \qquad T\in \cL(L^2(\T^n)), \ s\in \R^n. 
\end{equation*}
In particular, we have
\begin{equation} 
\alpha_s(U^k)= e^{is\cdot k} U^k, \qquad  k\in \Z^n. 
\end{equation}
This last property implies that the $C^*$-algebra $A_\theta$ is preserved by the action of $\R^n$. In fact, the induced action on $A_\theta$ is continuous, and so  the triple $(A_\theta, \R^n, \alpha)$ is a $C^*$-dynamical system.

The \emph{smooth noncommutative torus} $\cA_\theta$ is the sub-algebra of smooth elements of the $C^*$-dynamical system $(A_\theta, \R^n, \alpha)$. That is, 
\begin{equation*}
 \cA_\theta:=\left\{ u \in A_\theta; \ \alpha_s(u) \in C^\infty(\R^n; A_\theta)\right\}. 
\end{equation*}
In particular, all the unitaries $U^k$, $k\in \Z^n$, are contained in $\cA_\theta$. When $\theta=0$ we recover the algebra of smooth function on $\T^n$. The action of $\R^n$ on $A_\theta$ is infinitesimally generated by the canonical derivations $\dl_1, \ldots, \dl_n$ given by
\begin{equation*}
 \partial_j(u) = \partial_{s_j} \alpha_s(u)|_{s=0}, \qquad u\in \cA_\theta,  
\end{equation*}
In terms of the generators $U_1, \ldots, U_n$ we have 
 \begin{equation} 
 \partial_j(U_j)=iU_j \qquad \text{and} \qquad  \partial_j(U_k)=0 \ \text{when $j\neq k$}. 
  \label{eq:NCtori.delta-U}
\end{equation}
Moreover, we have 
\begin{equation}
  \partial_j(u^*)=\partial_j(u)^*, \qquad u\in \cA_\theta. 
    \label{eq:NCtori.derivation-involution} 
\end{equation}

The derivations $\dl_1, \ldots, \dl_n$ commute with each other. For $\alpha\in \N_0^n$ set $\dl^\alpha = \dl_1^{\alpha_1} \cdots \dl_n^{\alpha_n}$. The semi-norms $u \rightarrow \|\dl^\alpha(u)\|$ generate a locally convex space topology on $\cA_\theta$. This turns $\cA_\theta$ into a Fr\'echet $*$-algebra. This is even a \emph{good} Fr\'echet algebra in the sense that the set of invertible elements is open and the inverse map $u \rightarrow u^{-1}$ is continuous. As a result $\cA_\theta$ is closed under holomorphic functional calculus. This is even a pre-$C^*$-algebra in the sense of~\cite{GVF:Birkh01}, and so all the matrix algebras $M_m(\cA_\theta)$ are closed under holomorphic functional calculus as well. 

Let $\tau:\cL(L^2(\T^n))\rightarrow \C$ be the state defined by the constant function $1$, i.e., 
 \begin{equation*}
 \tau (T)= \scal{T1}{1}=(2\pi)^{-n}\int_{\T^n} (T1)(x) d x, \qquad T\in \cL\left(L^2(\T^n)\right).
\end{equation*}
In particular, we have
 \begin{equation*} 
\tau(1)=1, \qquad \tau(U^k)=0 \ \text{when $k\neq 0$}.  
 \end{equation*}
This induces a tracial state on the $C^*$-algebra $A_\theta$. Moreover (see, e.g.,~\cite{Ro:APDE08}), for $j=1,\ldots, n$,  we have
 \begin{gather}
 \tau\left[ \partial_j(u)\right] = 0 \qquad \forall u\in \cA_\theta, \label{eq:NCtori.closed-trace} \\
 \tau\left[ u\partial_j(v)\right] =- \tau\left[ \partial_j(u)v\right] \qquad \forall u,v\in \cA_\theta. 
 \label{eq:NCtori.integration-by-parts}
\end{gather}

Let $\scal{\cdot}{\cdot}$ be the sesquilinear form on $A_\theta$ defined by
\begin{equation}
 \scal{u}{v} = \tau\left(v^*u\right), \qquad u,v\in A_\theta. 
 \label{eq:NCtori.cAtheta-innerproduct}
\end{equation}
The family $\{ U^k; k \in \Z^n\}$ is orthonormal with respect to this sesquilinear form, and so we obtain a pre-inner product on the dense subspace $\cA_\theta^0$. The Hilbert space $\cH_\theta$ is the completion of $\cA_\theta^0$ with respect to this pre-inner product. Moreover, the multiplication of $\cA_\theta^0$ uniquely extends to a continuous bilinear map $A_\theta\times \cH_\theta \rightarrow \cH_\theta$ which provides us with a unital $*$-representation of $A_\theta$ into $\cL(\cH_\theta)$. This is precisely the GNS representation of $A_\theta$ associated with $\tau$. 

As $\{ U^k; k \in \Z^n\}$ is an orthonormal (Hilbert) basis, for every $u\in \cH_\theta$, we have a unique series decomposition in $\cH_\theta$, 
\begin{equation}
 u = \sum_{k \in \Z^n} u_k U^k, \qquad u_k:=\scal{u}{U^k}=\tau\big[(U^k)^*u\big]. 
 \label{eq:NCTori.Fourier-series}
\end{equation}
When $\theta=0$ we recover the usual Fourier series decomposition in $L^2(\T^n)$. The inclusion of $\cA_\theta^0$ into $\cH_\theta$ uniquely extends to continuous inclusions of $\cA_\theta$ and $A_\theta$ into $\cH_\theta$. In terms of the series decomposition~(\ref{eq:NCTori.Fourier-series}) the elements of $\cA_\theta$ are exactly the elements of $\cH_\theta$ for which the sequence of coefficients $(u_k)_{k\in \Z^n}$ has rapid decay. This gives a very concrete description of $\cA_\theta$. 

\section{Positivity and Hermitian Modules}\label{sec:positivity}
In this section, we review some facts on positive elements and Hermitian modules in the setting of smooth algebras. This also includes a discussion of Hermitian free modules over noncommutative tori. 

\subsection{Positivity and holomorphic functional calculus} In what follows we let  $\cA$ be a unital $*$-subalgebra of some $C^*$-algebra $A$ such that $\cA$ is closed under holomorphic functional calculus (e.g., $\cA$ is a pre-$C^*$-algebra in the sense of~\cite{GVF:Birkh01}). This implies that the invertible group $\cA^{-1}$ of $\cA$ agrees with $A^{-1}\cap\cA$, and so for all $x\in \cA$, we have 
\begin{equation}
 \Sp(x) = \{\lambda \in \C;\ x-\lambda\not\in A^{-1}\}= \{\lambda \in \C;\ x-\lambda\not\in \cA^{-1}\}.
 \label{eq:Pos.spectrum} 
\end{equation}

Let $A^+$  be the cone of positive elements of $A$, i.e., selfadjoint elements with non-negative spectrum. Recall that 
\begin{equation*}
 A^+= \{x^*x; \ x \in A\} = \{x^2;\ x \in A, \ x^*=x\}. 
\end{equation*}
As usual, given $x,y\in A$, we shall write $x\geq y$ when $x-y\in A^+$. In particular, $x\geq 0$ iff $x\in A^+$. More generally, if $c\in \R$, then $x\geq c$ iff $x$ is selfadjoint and its spectrum is contained in $[c,\infty)$. 

In what follows we denote by $\cA^{++}$ the set of positive elements of $\cA$ that are invertible. Thus, in view of~(\ref{eq:Pos.spectrum}), we have 
\begin{equation}
 \cA^{++}= \cA^{-1} \cap A^{+} = \{x\in \cA;\  x^*=x \ \text{and}\ \Sp(x)\subset (0,\infty)\}.
 \label{eq:Pos.cA++} 
\end{equation}

\begin{lemma}\label{lem:Pos.c A++-holomorphic}
 Let $x$ be a normal element of $\cA$.
\begin{enumerate}
   \item[(i)] If $f(z)$ is any given  holomorphic function near $\Sp(x)$ such that $f(\Sp(x))\subset (0,\infty)$, then $f(x)\in \cA^{++}$. 
    
    \item[(ii)] Assume that $\Sp(x)\subset (c_1, c_2)$, $c_i\in \R$. Then there are $y_1$ and $y_2$ in $\cA$ such that
    \begin{equation*}
         c_1+y_1^*y_1 = x= c_2-y_2^*y_2.     
    \end{equation*}
 \end{enumerate}
\end{lemma}
\begin{proof}
 Let $f(z)$ be a  holomorphic function near $\Sp(x)$ such that $f(\Sp(x))\subset (0,\infty)$. Then $f(x)$ is normal and has positive spectrum. Moreover, $f(x)\in \cA$ since $\cA$ is closed under holomorphic functional calculus. It then follows from~(\ref{eq:Pos.cA++}) that $f(x)\in \cA^{++}$. This proves (i). 
 
 Suppose that $\Sp(x)\subset (c_1, c_2)$. Then $\Sp(x-c_1)\subset (0,\infty)$, and so the function $z \rightarrow \sqrt{z}$ is holomorphic near $\Sp(x-c_1)$ and maps $\Sp(x-c_1)$ to $(0,\infty)$. 
 Thus, by the first part $y_1:=\sqrt{x-c_1}$ is an element of $\cA^{++}$. 
 In particular, this is a selfadjoint element of $\cA$ such that $x-c_1=(y_1)^2=y_1^*y_1$, and hence $x=c_1+y_1^*y_1$. 
 Likewise, as $\Sp(c_2-x)\subset (0,\infty)$, if we set $y_2=\sqrt{c_2-x}$, then $y_2\in \cA^{++}$ and $c_2-x=(y_2)^2=y_2^*y_2$. Thus, $x=c_2-y_2^*y_2$. The proof is complete. 
\end{proof}

In this paper we will make use of the following consequence of Lemma~\ref{lem:Pos.c A++-holomorphic}. 

\begin{lemma}\label{lem:Pos.cA++}
 Let $x\in \cA$. Then the following are equivalent: 
\begin{itemize}
 \item[(i)] $x\in \cA^{++}$. 
 
 \item[(ii)] There is $c>0$ such that $x\geq c$. 
 
 \item[(iii)] There are $y\in \cA$ and $c>0$ such that $x=y^*y+c$. 
\end{itemize}
\end{lemma}
\begin{proof}
 As (ii) precisely means that $x$ is selfadjoint and has positive spectrum, the equivalence of (i) and (ii) is an immediate consequence of~(\ref{eq:Pos.cA++}). It is also immediate that (iii) implies (ii). Conversely, if $x\geq c$ with $c>0$, then $x$ is selfadjoint and its spectrum is contained in $[c,\infty)$. Therefore, by Lemma~\ref{lem:Pos.c A++-holomorphic} there are $y\in \cA$ and $c>0$ such that $x=y^*y+c$. Thus (ii) implies (iii). The proof is complete. 
\end{proof}

\begin{remark}
 It follows from (ii) that, if $x\in \cA^{++}$, then $x+y\in \cA^{++}$ for all $y\in \cA$, $y\geq 0$. 
\end{remark}

\begin{remark}
 It follows from~(iii) that $\cA^{++}$ does not depend on the embedding of $\cA$ into $A$. Moreover, as the proof above shows, in (iii) we may take $y$ to be in 
 $\cA^{++}$.  
\end{remark}

\subsection{Hermitian Modules} 
Let $\cE$ be a left module over $\cA$. The dual $\cE^*:=\Hom_\cA(\cE,\cA)$ inherits a right action $\cE^*\times \cA\ni  (\omega, x) \rightarrow  \omega x\in \cE^*$, where $\omega x$ is defined by
\begin{equation}
 \acou{\omega x}{\xi}:= \acou{\omega}{\xi}x, \qquad \xi \in \cE. 
 \label{eq:Pos.dual-right-action}
\end{equation}
It will be convenient to think of $\cE^*$ as a left module over the opposite algebra $\cA^{\opp}$. That is, the vector space $\cA$ equipped with the opposite product, 
\begin{equation*}
 x\circ y := yx, \qquad x,y\in \cA. 
\end{equation*}
The involution of $\cA$ is an anti-linear anti-automorphism of $\cA^\opp$ as well, and so $\cA^\opp$ is a $*$-algebra. Moreover, $\cA^\opp$ is closed under holomorphic functional calculus in the opposite $C^*$-algebra $A^\opp$. In what follows, we denote by $(x, \omega)\rightarrow x^\opp \omega$ the left-action of $\cA^\opp$ corresponding to the right-action~(\ref{eq:Pos.dual-right-action}). That is, 
\begin{equation*}
 \acou{x^\opp\omega}{\xi}= \acou{\omega x}{\xi}= \acou{\omega}{\xi}x = x\circ \acou{\omega}{\xi}. 
\end{equation*}

\begin{definition}
 A Hermitian metric on a left $\cA$-module $\cE$ is a map $\acoup\cdot\cdot:\cE\times \cE\rightarrow \cA$ satisfying the following properties:
 \begin{itemize}
 \item[(i)] It is \emph{$\cA$-sesquilinear}, i.e., $\acoup{x\xi}{y\eta}=x\acoup{\xi}{\eta}y^*$ for all $\xi,\eta\in \cE$ and $x,y\in \cA$. 
 
 \item[(ii)] It is \emph{positive}, i.e., $\acoup\xi\xi\geq 0$ for all $\xi\in \cE$. 
 
 \item[(iii)] It is \emph{non-degenerate}, i.e.,  the map $\cE\ni \xi \rightarrow \acoup{\cdot}\xi\in \cE^*$ is an anti-linear isomorphism. 
\end{itemize}
\end{definition}

\begin{remark}
 The conditions (i) and (ii) imply that 
\begin{equation*}
 \acoup\eta \xi = \acoup\xi{\eta}^* \qquad \text{for all $\xi,\eta\in \cE$}. 
\end{equation*}
\end{remark}

\begin{remark}
 Anti-linearity in (iii) is meant in the following sense. Given $\xi \in \cE$ set $\xi^*:=\acoup{\cdot}{\xi}\in \cE^*$, so that $\acou{\xi^*}{\eta}=\acoup{\eta}{\xi}$ for all $\eta \in \cE$. Then,  we have
 \begin{equation*}
 (x \xi)^* = \xi^* x^*= (x^*)^\opp \xi^* \qquad \text{for all $x\in \cA$}. 
\end{equation*}
\end{remark}

\begin{definition}
 A Hermitian module over $\cA$ is a left module $\cE$ over $\cA$ equipped with a Hermitian metric. 
\end{definition}

\begin{example}
 Given any $m\geq 1$, the free module $\cA^m$ is a Hermitian module with respect to canonical Hermitian metric, 
  \begin{equation}
 \acoup{\xi}{\eta} = \sum_{1\leq j \leq m} \xi_j \eta_j^*, \qquad \xi=(\xi_j), \ \eta=(\eta_j). 
 \label{eq:Positivity.canonical-metric}
\end{equation}
By construction this is a positive sesquilinear map. It is also non-degenerate, since the map $\xi\rightarrow \acoup{\cdot}{\xi}$ sends the canonical basis of $\cA^m$ to its dual basis. 
\end{example}

Let $(\cE, \acoup\cdot\cdot)$ be a Hermitian $\cA$-module. Given $\xi\in \cE$ set $\xi^*=\acoup{\cdot}{\xi}\in \cE^*$. As above we regard $\cE^*$ as a left $\cA^\opp$-module. By assumption $\xi \rightarrow \xi^*$ is an anti-linear isomorphism from $\cE$ onto $\cE^*$.  Thus, given any $\omega\in \cE^*$, there is a unique $\xi_\omega\in \cE$ 
such that $\omega=\xi_\omega^*$. Let  $\acoup\cdot{\cdot}':\cE^*\times \cE^*\rightarrow \cA$ be the map defined by
\begin{equation}
 \acoup{\omega}{\zeta}' := \acoup{\xi_{\zeta}}{\xi_\omega}, \qquad \omega,\zeta \in \cE^*. 
 \label{eq:Pos.dual-Hermitian}
\end{equation}
Equivalently, we have 
\begin{equation*}
 \acoup{\xi^*}{\eta^*}'= \acoup{\xi}{\eta}^* \qquad \text{for all $\xi,\eta\in \cE$}. 
\end{equation*}

\begin{lemma}\label{lem:Pos.dual-Hermitian}
$(\cE^*, \acoup\cdot{\cdot}')$ is a Hermitian $\cA^\opp$-module. 
\end{lemma}
\begin{proof}
We just need to check that $ \acoup\cdot{\cdot}'$ is a Hermitian metric. The sesquilinearity of $\acoup{\cdot}{\cdot}$ and the anti-linearity of the map $\omega \rightarrow \xi_\omega$ imply that $(\cdot,\cdot)'$ is $\cA^\opp$-sesquilinear. Indeed, as $\xi_{x^\opp \omega}=x^*\xi_\omega$, for all $\omega, \zeta\in \cE^*$ and $x,y\in \cA$, we have
\begin{equation*}
 \acoup{x^\opp \omega}{y^\opp \zeta}' = \acoup{\xi_{x^\opp \omega}}{\xi_{y^\opp \zeta}}^*= \acoup{x^*\xi_\omega}{y^*\xi_{\zeta}}^*. 
\end{equation*}
 As $\acoup{x^*\xi_\omega}{y^*\xi_\zeta}=x^* \acoup{\xi_\omega}{\xi_{\zeta}}y= x^* [\acoup{\omega}{\zeta}']^*y= [y^*\acoup{\omega}{\zeta}' x]^*$, we see that
\begin{equation*}
\acoup{x^\opp \omega}{y^\opp \zeta}' = y^* \acoup{\omega}{\zeta}' x = x\circ \acoup{\omega}{\zeta}' \circ y^*. 
\end{equation*}
This shows that $\acoup{\cdot}{\cdot}$  is $\cA^\opp$-sesquilinear. Moreover, the positivity of $\acoup{\cdot}{\cdot}$ implies that, for all $\omega\in \cE^*$, we have 
$\acoup{\omega}{\omega}'= \acoup{\xi_\omega}{\xi_\omega}^*= \acoup{\xi_\omega}{\xi_\omega}\geq 0$. That is, $\acoup{\cdot}{\cdot}'$ is positive. 

It remains to show that $\acoup{\cdot}{\cdot}'$ is non-degenerate. Let $\cE^{**}=\Hom_{\cA^\opp}(\cE^*,\cA^\opp)$ be the bidual. This is a left module over $\cA$ with respect to the action $\cA\times \cE^{**}\ni (x,\zeta)\rightarrow x \zeta \in \cE^{**}$ given by
\begin{equation*}
 \acou{x\zeta}{\omega}= \acou{\zeta}{\omega} \circ x =x  \acou{\zeta}{\omega}, \qquad \omega\in \cE^*. 
\end{equation*}
Given any $\omega\in \cE^*$ we set $\omega^*:=\acoup{\cdot}{\omega}'\in \cE^{**}$. We need to show that the map $\Psi:\omega \rightarrow \omega^*$ is an anti-linear isomorphism from $\cE^*$ onto $\cE^{**}$. To see this let us denote by $\Phi$ the isomorphism $\xi \rightarrow \xi^*$ from $\cE$ onto $\cE^*$ defined by 
$\acoup{\cdot}{\cdot}$. By duality we get an anti-linear transpose map $\Phi^t: \cE^{**}\rightarrow \cE^*$ given by
\begin{equation*}
 \acou{\Phi^t(\zeta)}{\eta}:=\acou{\zeta}{\Phi(\eta)}^*= \acou{\zeta}{\eta^*}^*, \qquad \zeta\in \cE^{**},\ \eta \in \cE. 
\end{equation*}
By functoriality we get an anti-linear isomorphism with inverse the transpose of $\Phi^{-1}$. 

Bearing this in mind, let $\omega\in \cE^*$ and $\eta \in \cE$. Then we have
\begin{equation*}
 \acou{\Phi^t \circ \Psi(\omega)}{\eta}= \acou{\omega^*}{\eta^*}^* = \left[\acoup{\eta^*}{\omega}'\right]^*= \acoup{\xi_{\eta^*}}{\xi_\omega}= \acoup{\eta}{\xi_\omega}= \acou{\xi_\omega^*}{\eta}=\acou{\omega^*}{\eta}. 
\end{equation*}
This shows that $\Phi^t \circ \Psi=\op{id}$, and so $\Psi= (\Phi^t)^{-1}$. Thus,  $\Psi$ is an isomorphism, and hence $\acoup{\cdot}{\cdot}'$ is non-degenarate. It then follows that $\acoup{\cdot}{\cdot}'$ is a Hermitian metric. The proof is complete. 
\end{proof}

\begin{remark}
A standard elaboration of the arguments of the proof above shows that $\Psi \circ \Phi$ agrees with canonical map from $\cE$ to $\cE^{**}$, and so $\cE$ and its bidual are isomorphic left modules. 
\end{remark}

\subsection{Free Hermitian modules over $\cA_\theta$} 
Given $m\geq 2$ let $M_m(\cA_\theta)$ be the $*$-algebra of $m\times m$-matrices with entries in $\cA_\theta$. The $*$-action of $\cA_\theta$ on $\cH_\theta$ naturally gives to a unitary action of $M_m(\cA_\theta)$ on the Hilbert space $\cH^m_\theta=\cH_\theta \oplus \cdots \oplus \cH_\theta$ ($m$-summands). Equivalently, $\cH^m_\theta$ is just the completion of $\cA^m_\theta$ with respect to the pre-inner product, 
\begin{equation}
 \scal{\xi}{\eta}= \sum_{1\leq i \leq m}  \scal{\xi_i}{\eta_i}= \sum_{1\leq i \leq m} \tau\left[\xi_i \eta_i^*\right]=\tau \left[ \acoup{\xi}{\eta}\right], \qquad \xi,\eta \in \cA^m_\theta,
 \label{eq:Pos.inner product-Atm}
\end{equation}
where $\acoup{\cdot}{\cdot}$ is the canonical Hermitian metric~(\ref{eq:Positivity.canonical-metric}). All this allows us to regard $M_m(\cA_\theta)$ as a $*$-subalgebra of the $C^*$-algebra $\cL(\cH^m)$. In addition, $M_m(\cA_\theta)$ is closed under holomorphic calculus (see, e.g., \cite[Proposition~3.39]{GVF:Birkh01}).

We what follows we denote by $\GL_m(\cA_\theta)$ the invertible group of $M_m(\cA_\theta)$ and by $\GL^+_m(\cA_\theta)$ the set of invertible positive elements. As $M_m(\cA_\theta)$ is closed under holomorphic calculus (see, e.g., \cite[Proposition~3.39]{GVF:Birkh01}), Lemma~\ref{lem:Pos.cA++} provides us with characterizations of  $\GL^+_m(\cA_\theta)$. 

Note also that matrix multiplication gives rise to a right action  $(\xi, a)\rightarrow \xi a$ of $M_m(\cA_\theta)$ on the free $\cA_\theta$-module $\cA_\theta^m$, i.e., we get a left-action of the opposite algebra $M_m(\cA_\theta)^\opp$. Namely, given any $\xi=(\xi_j)\in \cA_\theta^m$ and $a=(a_{ij})\in M_m(\cA_\theta)$, we have 
\begin{equation*}
 \xi a = \big( \sum_i \xi_i a_{ij} \big)_{1\leq i \leq m} . 
\end{equation*}
If we let $a^*=(a_{ji}^*)$ be the adjoint matrix of $a$, then, for all $\xi,\eta\in \cA_\theta^m$, we have
\begin{equation*}
 \acoup{\xi a}\eta = \sum_{1 \leq i,j\leq n} \xi_i a_{ij} \eta_j^* = \sum_{1\leq i \leq n} \big( \sum_{1\leq j \leq n} \eta_j a_{ij}^*\big) =\acoup{\xi}{\eta a^*}.
\end{equation*}
In particular, we see that 
\begin{equation*}
 a=a^* \ \Longleftrightarrow \ \acoup{\xi a}{\eta}= \acoup{\xi}{\eta a} \ \text{for all $\xi,\eta \in \cA_\theta$}. 
\end{equation*}

In what follows we denote by $(\epsilon_1, \ldots, \epsilon_m)$ the canonical basis of $\cA_\theta^m$. Any Hermitian metric $(\cdot,\cdot)_1$ on $\cA_\theta^m$ is uniquely determined by its coefficent matrix $h=(h_{ij})$, where $h_{ij} =(\epsilon_i,\epsilon_j)$, $1\leq i,j\leq m$. Indeed, by sesquilinearity, for all $\xi = (\xi_i)$ and $\eta=(\eta_i)$ in $\cA_\theta^m$, we have 
\begin{equation}
(\xi,\eta)_1 = \sum_{1\leq i,j\leq m} (\xi_i \epsilon_i,\eta_j \epsilon_j)_1 =  \sum_{1\leq i,j\leq m} \xi_i( \epsilon_i, \epsilon_j)_1 \eta_j^*=  
\sum_{1\leq i,j\leq m} \xi_ih_{ij} \eta_j^* = \acoup{\xi h}{\eta}. 
 \label{eq:Pos.brak-h}
\end{equation}

\begin{proposition}\label{prop:Positivity.Hermitian-metrics-free}
Let $(\cdot,\cdot)_1$ be a Hermitian metric on $\cA_\theta^m$. Then its coefficient matrix is in $\GL_m^+(\cA_\theta)$. 
Conversely, any $h\in \GL_m^+(\cA_\theta)$ defines a Hermitian metric $\acoup{\cdot}{\cdot}_h$ on $\cA_\theta^m$ given by 
 \begin{equation}
 \acoup{\xi}{\eta}_h:= \acoup{\xi h}{\eta}= \sum_{1 \leq i,j\leq n} \xi_i h_{ij} \eta_j^*, \qquad \xi=(\xi_j), \ \eta=(\eta_j).
 \label{eq:Positivity.acouph}
\end{equation}
This gives a one-to-one correspondence between $\GL_m^+(\cA_\theta)$ and Hermitian metrics on $\cA_\theta^m$. 
\end{proposition}
\begin{proof}
Let $(\cdot,\cdot)_1$ be a Hermitian metric on $\cA_\theta^m$ and $h=(h_{ij})$ its coefficient matrix, with $h_{ij} =(\epsilon_i,\epsilon_j)_1$.  
Note that $h_{ji}=(\varepsilon_j, \varepsilon_i)_1= (\varepsilon_i, \varepsilon_j)_1^*=h_{ij}^*$, and so $h^*=h$, i.e., $h$ is selfadjoint.  
In addition, let $\xi=(\xi_i)$ be in $\cA_\theta^m$, and set $\xi^*=(\xi_i^*)$. By using~(\ref{eq:Pos.inner product-Atm})--(\ref{eq:Pos.brak-h}) and the fact that 
$\tau$ is a trace we get
\begin{equation*}
 \scal{h\xi}{\xi}=  \tau \left[ \acoup{h\xi}{\eta}\right]= \tau \big[ \sum_{i,j} h_{ij}\xi_j \xi_i^* \big] =  \tau \big[ \sum_{i,j} \xi_i^* h_{ij}\xi_j  \big] = \tau \left[ \big(\xi^*, \xi^*\big)\right].
\end{equation*}
The positivity of $(\cdot,\cdot)_1$ and $\tau$ then ensures us that $ \scal{h\xi}{\xi}\geq 0$. Combining this with the density of $\cA_\theta^m$ in $\cH^m_\theta$ shows that $ \scal{h\xi}{\xi}\geq 0$ for all $\xi \in \cH^m_\theta$. That is, $h$ is positive. 

Let us denote by $\Psi$ the anti-linear isomorphism $\cA_\theta^m \ni \xi \rightarrow (\cdot, \xi)_1\in (\cA_\theta^m)^*$ defined by the Hermitian metric $(\cdot, \cdot)_1$. We also denote by $\Phi$ the canonical isomorphism $\xi \rightarrow \xi^*=\acou{\cdot}{\xi}$. As $h$ is selfadjoint, given any $\xi$ and $\eta$ in $\cA_\theta^m$, we have 
\begin{equation}
 \acou{\Psi(\xi)}{\eta} = \big(\eta,\xi\big)_1=(\eta h,\xi) =(\eta, \xi h)= \acou{\Phi(\xi h)}{\eta}. 
 \label{eq:Positivity.Phi-Psi-h}
\end{equation}
 Thus, if we denote by $h^\opp$ the right action of $h$ on $\cA_\theta^m$, then $\Psi= \Phi \circ h^\opp$, i.e., $h^\opp=\Phi^{-1}\circ \Psi$. In particular, $h^\opp$ is an invertible left module map, and so $h$ is invertible in $M_m(\cA_\theta)$ with inverse $h^{-1}=(h^{ij})$, where $h^{ij}:=\acoup{\Psi^{-1}\circ \Phi(\varepsilon_i)}{\varepsilon_j}$, $i,j=1,\ldots, m$. This shows that $h$ is invertible and positive, i.e., it lies in $\GL_m^+(\cA_\theta)$. 

Conversely, let $h  \in \GL_m^+(\cA_\theta)$. It is immediate that the map $\acoup{\cdot}{\cdot}_h$ defined by~(\ref{eq:Positivity.acouph}) is $\cA_\theta$-sesquilinear. Moreover, by Lemma~\ref{lem:Pos.cA++} there are $c>0$ and $b\in M_m(\cA_\theta)$ such that $h=bb^*+c$. Thus, given any $\xi \in M_m(\cA_\theta)$, we have 
\begin{equation}
 \acoup{\xi}{\xi}_h= \acoup{\xi h}{\xi}= \acoup{\xi bb^*}{\xi}+c\acoup{\xi}{\xi}=\acoup{\xi b}{\xi b} + c\acoup{\xi}{\xi}\geq c(\xi,\xi)\geq 0.
 \label{eq:Pos.lower-bound} 
\end{equation}
In particular, this shows that  $\acoup{\cdot}{\cdot}_h$ is positive.  

Let us now show that  $\acoup{\cdot}{\cdot}_h$ is non-degenerate. Given $\xi\in \cA_\theta^m$, set $\xi_h^*=\acoup{\cdot}{\xi}\in (\cA_\theta^m)^*$. In the same way as in~(\ref{eq:Positivity.Phi-Psi-h}), for all $\eta \in \cA_\theta^m$, we have 
\begin{equation}
 \acou{\xi^*_h}{\eta}=\acoup{\eta}{\xi}_h= \acoup{\eta h}{\xi}= \acoup{\eta }{h\xi}=\acou{(\xi h)^*}{\eta}.
 \label{eq:Pos.xi*-xi*h}  
\end{equation}
This shows that the map $\xi \rightarrow \xi_h^*$ is the composition of the right action by $h$ on $\cA_\theta^m$ with the anti-linear isomorphism $\xi \rightarrow \xi^*$. The right action by $h$ is an invertible left module map, since $h$ is invertible. Therefore, the map $\xi \rightarrow \xi_h^*$ is an isomorphism, and so $\acoup{\cdot}{\cdot}_h$ is non-degenerate. As this is a positive sesquilinear map, we see that $\acoup{\cdot}{\cdot}_h$ is a Hermitian metric.  The proof is complete. 
\end{proof}

\begin{corollary}\label{cor:Positivity.equivalence-free}
 Let  $(\cdot,\cdot)_1$ be a Hermitian metric on $\cA_\theta^m$. 
 \begin{enumerate}
       \item[(i)]  There are $c_1>0$ and $c_2>0$ such that
            \begin{equation}
               c_1\acoup{\xi}{\xi} \leq (\xi,\xi)_1 \leq  c_2\acoup{\xi}{\xi} \qquad \forall \xi \in \cA_\theta^m.
               \label{eq:Positivity.equivalence-free} 
            \end{equation}
            
            \item[(ii)] $(\xi,\xi)_1\in \cA_\theta^{++}$ for all $\xi \in \C^m\setminus 0$. 
\end{enumerate}
\end{corollary}
\begin{proof}
 Let $h=(h_{ij})$ be the coefficient matrix of the Hermitian metric $(\cdot,\cdot)_1$. We know by Proposition~\ref{prop:Positivity.Hermitian-metrics-free} that $h\in \GL_m^+(\cA_\theta)$. That is,  $h$ is selfadjoint and has positive spectrum. In particular, $\Sp (h)\subset (c_1,c_2)$ with $0<c_1<c_2$. Thus, by Lemma~\ref{lem:Pos.c A++-holomorphic} there are $b_i\in M_m(\cA_\theta)$, $i=1,2$, such that $c_1+b_1b_1^*=h=c_2-b_2b_2^*$. Let $\xi\in \cA_\theta^m$. The equality  $h=c_1+b_1b_1^*$ implies that we have 
\begin{equation}
 (\xi,\xi)_1=\acoup{\xi h}{\xi}= c_1 \acoup{\xi}{\xi}+ \acoup{\xi b_1b_1^*}{\xi}= c_1 \acoup{\xi}{\xi} + \acoup{\xi b_1}{\xi b_1}\geq c_1 \acoup{\xi}{\xi}.
 \label{eq:Positivity.metric-c2} 
\end{equation}
 Likewise, by using the equality $h=c_2- b_2b_2^*$  we get $ (\xi,\xi)_1= \acoup{\xi h}{\xi}=c_2 \acoup{\xi}{\xi} - \acoup{\xi b_2}{\xi b_2}$, and so  
 $(\xi,\xi)_1 \leq  c_2\acoup{\xi}{\xi}$. This gives~(\ref{eq:Positivity.equivalence-free}). 
 
 In order to get (ii) we just note that  if $\xi \in \C^m\setminus 0$, then~(\ref{eq:Positivity.metric-c2}) gives $(\xi,\xi)_1\geq c_1 \acoup{\xi}{\xi}=c_1|\xi|^2$. As $c_1|\xi|^2>0$ it then follows from Lemma~\ref{lem:Pos.cA++} that $(\xi,\xi)_1\in \cA_\theta^{++}$. The proof is complete. 
 \end{proof}

Given any left module $\cE$ over $\cA_\theta$,  we shall say that two Hermitian metrics $\acoup{\cdot}{\cdot}_1$ and $\acoup{\cdot}{\cdot}_2$ on $\cE$ are equivalent when there are constants $c_1>0$ and $c_2>0$ such that 
\begin{equation*}
  c_1\acoup{\xi}{\xi}_1 \leq \acoup{\xi}{\xi}_2 \leq  c_2\acoup{\xi}{\xi}_1 \qquad \forall \xi \in \cE. 
\end{equation*}
For instance, the first part of Corollary~\ref{cor:Positivity.equivalence-free} asserts that all Hermitian metrics on $\cA_\theta^m$ are equivalent to the canonical Hermitian metric~(\ref{eq:Positivity.canonical-metric}). More generally, we have the following. 

\begin{corollary}\label{cor:Positivity.equivalence-proj}
 Assume that $\cE$ is a finitely generated left module over $\cA_\theta$. Then all Hermitian metrics on $\cE$ are equivalent. 
\end{corollary}
\begin{proof}
 Without any loss of generality we may assume that $\cE=\cA_\theta^m e$, with $e\in M_m(\cA_\theta)$, $e^2=e$, $m\geq 1$. It is then enough to observe  that every Hermitian metric on $\cA_\theta^m e$ is the restriction of a Hermitian metric on $\cA_\theta^m$, and hence is equivalent to the restriction of the canonical Hermitian metric. The proof is complete. 
\end{proof}

\section{Riemannian Metrics on Noncommutative Tori} \label{sec:Riem}
In this section,  we recall the definition of Riemannian metrics on noncommutative tori by Rosenberg~\cite{Ro:SIGMA13}. Our definition slightly differs from Rosenberg's original definition. We also present a few examples and introduce a notion of conformal equivalence of Riemannian metrics. 

\subsection{Riemannian metrics} 
The left action of $\cA_\theta$ on itself gives rise to a left action on the algebra $\cL(\cA_\theta)$ of continuous endomorphisms on $\cA_\theta$. This allows us to regard $\cL(\cA_\theta)$ as a left $\cA_\theta$-module. 
Note that the  canonical derivations $\partial_1, \ldots, \partial_n$ are linearly independent in $\cL(\cA_\theta)$. Indeed, if $\sum_j a^j \partial_j=0$ with $a^j\in \cA_\theta$, then~(\ref{eq:NCtori.delta-U}) ensures us that, for $j=1,\ldots, n$, we have $ 0 = \sum_{l} a^l \delta_l(U_j)=ia^jU_j$, and hence $a^1=\cdots =a^n=0$.

\begin{definition}[\cite{Ro:SIGMA13}] 
 $\cX_\theta$ is the free left $\cA_\theta$-module generated by the derivations $\partial_1, \ldots, \partial_n$. 
\end{definition}

We shall think of $\cX_\theta$ as the module of vector fields on $\cA_\theta$. The coordinate system on  $\cX_\theta$ defined by $\partial_1, \ldots, \partial_n$ gives rise to an $\cA_\theta$-module isomorphism $\cA_\theta^n\simeq \cX_\theta$ under which $(\partial_1, \ldots, \partial_n)$ corresponds to the canonical basis of $\cA^n_\theta$. Under this identification we have a one-to-one correspondence between Hermitian metrics on $\cA_\theta^m$ and $\cX_\theta$. Combining this with 
Proposition~\ref{prop:Positivity.Hermitian-metrics-free} we then obtain a one-to-one correspondence between $\GL^+_n(\cA_\theta)$ and Hermitian metrics on $\cX_\theta$. Namely, any Hermitian metric $\acoup{\cdot}{\cdot}$ is uniquely determined by its coefficient matrix $h=(h_{ij})$, with $h_{ij}=\acoup{\dl_i}{\dl_j}$, so that, for all $X=\sum_{i} X^{i} \dl_i$ and $Y=\sum_{i} Y^i \dl_i$ in $\cX_\theta$, we have 
\begin{equation}
 \acoup{X}{Y}= \sum_{i,j} X^{i} \acoup{\dl_i}{\dl_j} (Y^j)^* = \sum_{i,j} X^{i} h_{ij} (Y^j)^*. 
 \label{eq:Riem.Hermitian-h} 
\end{equation}
The coefficient matrix $h$ is an element of $\GL^+_n(\cA_\theta)$. Conversely, any $h \in \GL^+_n(\cA_\theta)$ defines this way a Hermitian metric  $\acoup{\cdot}{\cdot}_h$ on $\cX_\theta$. 

On an ordinary manifold $M$ a Riemannian metric on a given manifold $M$ is a Hermitian metric $\acoup{\cdot}{\cdot}$ on the complexified tangent space $T_\C M$ that takes real values on real vector fields. Equivalently, in any local coordinates, the coefficients $\acoup{\dl_{x_i}}{\dl_{x_j}}$ are real-valued. 

The analogue of the above condition for a Hermitian metric $\acoup{\cdot}{\cdot}$ on $\cX_\theta$ is requiring the coefficients $\acoup{\dl_{i}}{\dl_{j}}$ to be selfadjoint. That is, the matrix of the Hermitian metric has selfadjoint entries. This yields the definition of a Riemannian metric in~\cite{Ro:SIGMA13}. There is a slight issue with this condition since we would like the inverse matrix to satisfy the same condition, so as to have a dual Hermitian metric on forms to have selfadjoint values on ``real'' 1-forms (see Section~\ref{sec:1-forms}). As it turns out, the inverse of a positive invertible matrix with selfadjoint entries need not have selfadjoint entries. For instance, consider a $2\times 2$-matrix of the form, 
\begin{equation*}
 h= \begin{pmatrix}
 1 & a \\
 0 & b
\end{pmatrix}^*\begin{pmatrix}
 1 & a \\
 0 & b
\end{pmatrix}= 
\begin{pmatrix}
 1 & a \\
 a & a^2+b^2
\end{pmatrix},
\end{equation*}where $a$ and $b$ are selfadjoint elements of $\cA_\theta$ and $b$ is invertible. Then $h$ has inverse
\begin{equation*}
 h^{-1} = 
\begin{pmatrix}
 1 + ab^{-2} a&  -ab^{-2} \\
  -b^{-2}a & b^{-2}
\end{pmatrix}
\end{equation*}
In particular, the off-diagonal entries of $h^{-1}$ cannot be selfadjoint when $[a,b]\neq 0$. 

In what follows we denote by $\cA_\theta^\R$ the real subspace of selfadjoint elements of $\cA_\theta$ and denote by $M_m(\cA_\theta^\R)$, $m\geq 2$, the real space of $m\times m$ matrices with entries in $\cA_\theta^\R$. 

\begin{definition}
 $\GL_m(\cA_\theta^\R)$, $m\geq 2$,  consists of matrices $h\in \GL_m(\cA_\theta)$ such that $h$ and $h^{-1}$ both have selfadjoint entries.  
\end{definition}

\begin{definition}
$\GL_m^+(\cA_\theta^\R)$, $m\geq 2$,  consists of matrices in $\GL_m(\cA_\theta^\R)$ that are positive.   
\end{definition}

\begin{remark}\label{rmk:selfadjoint-Hermitian-symmetric}
 If $a=(a_{ij})\in M_m(\cA_\theta^\R)$ and $a$ is selfadjoint, then $a_{ji}=a_{ij}^*=a_{ij}$, i.e., the matrix $a$ is symmetric. In particular, all the elements of 
 $\GL^{+}_m(\cA_\theta^\R)$ are symmetric matrices. 
\end{remark}

\begin{definition}[compare~\cite{Ro:SIGMA13}] 
 A Riemannian metric on $\cA_\theta$ is any Hermitian metric on $\cX_\theta$ whose coefficient matrix is in $\GL_n^+(\cA_\theta^\R)$. 
\end{definition}

We have a one-to-one correspondence between $\GL_n^+(\cA_\theta^\R)$ and Riemannian metrics on $\cA_\theta$ given by~(\ref{eq:Riem.Hermitian-h}). As above, given any $g\in\GL_n^+(\cA_\theta^\R)$, 
we shall denote by $\acoup{\cdot}{\cdot}_g$ the corresponding Riemannian metric, i.e., 
\begin{equation*}
  \acoup{X}{Y}_g= \sum_{i,j} X^{i} g_{ij} (Y^j)^*, \qquad X=\sum_{i} X^{i} \dl_i, \ Y=\sum_{i} Y^i \dl_i.
\end{equation*}
We shall use  this correspondence to identify Riemannian metrics and their matrices. 

On an ordinary manifold $M$ two Riemannian metrics $g$ and $\hat{g}$  are conformally equivalent when they define same angles between vectors in each tangent space $T_x M$, $x\in M$. Equivalently, there is a function $k(x)\in C^\infty(M)$ such that $\hat{g}=k(x)^{2} g$.  

\begin{definition}
 We say that two Riemannian metrics with respective matrices $g=(g_{ij})$ and $\hat{g}=(\hat{g}_{ij})$ are \emph{conformally equivalent} when there is $k\in \cA_\theta^{++}$ such that $\hat{g}=k g k=(kg_{ij}k)$. 
\end{definition}

\subsection{Examples} 
Let us now look at some examples of Riemannian metrics. 

\subsubsection*{Conformal deformations of flat metrics} The Euclidean metric is $g_{ij}=\delta_{ij}$. Connes-Tretkoff~\cite{CT:Baltimore11} considered conformal deformations of this metric, i.e., metrics of the form, 
\begin{equation}
 g_{ij}=k^2 \delta_{ij}, \qquad k\in \cA_\theta^{++}. 
 \label{eq:Riem.conf-def-Euclidean}
\end{equation}
Such metrics have been considered in various subsequent papers as well (see, e.g., \cite{CM:JAMS14, DGK:arXiv18, Fa:JMP15, FK:JNCG12, FK:LMP13,  FK:JNCG13, FK:JNCG15, FGK:JNCG19, LM:GAFA16,  Liu:arXiv18a, Liu:arXiv18b}). As in~\cite{GK:arXiv18} we may also consider conformal 
deformations of more general flat metrics, 
\begin{equation}
 g_{ij}=k^2 g^0_{ij}, \qquad (g^0_{ij})\in \GL^+_n(\R), \quad k\in \cA^{++}. 
 \label{eq:Riem.conf-def-flat}
\end{equation}

\subsubsection*{Product metrics}
A product metric is of the form, 
\begin{equation*}
 g = 
\begin{pmatrix}
 g^{(1)} & 0\\
 0 & g^{(2)}
\end{pmatrix}, \qquad g^{(j)}\in \GL_{m_j}^+(\cA_\theta^\R). 
\end{equation*}
A special class of such metrics are products of conformal deformations of Euclidean metrics, 
\begin{equation*}
 g= \begin{pmatrix}
k_1^2 I_{m_1}  & 0\\
 0 & k_2^2 I_{m_2} 
\end{pmatrix}, \qquad k_j\in \cA_\theta^{++}, \quad k_1\neq k_2. 
\end{equation*}
Here $I_{m_j}$ is the $m_j\times m_j$-identity matrix. Similar kind of metrics have been considered in~\cite{CF:MJM19, DGK:arXiv18, DS:SIGMA15, KS:JMP18}. 

\subsubsection*{Self-compatible Riemannian metrics} 

\begin{definition}\label{def:Riem.self-compatible}
 Two matrices $a=(a_{ij})\in M_p(\cA_\theta)$ and $b=(b_{kl})\in M_q(\cA_\theta)$ are \emph{compatible} when each entry of $a$ commutes with every entry of $b$, i.e., $[a_{ij}, b_{kl}]=0$ for $i,j=1,\ldots p$ and $k,l=1,\ldots, q$. A matrix  $a\in M_m(\cA_\theta)$  is \emph{self-compatible} when it is compatible with itself. 
\end{definition}

\begin{lemma}\label{lem:compatible-matrices}
 Let $a=(a_{ij})\in M_p(\cA_\theta)$ and $b=(b_{kl})\in M_q(\cA_\theta)$ be compatible. Then, for every holomorphic function $f_1$ near $\Sp(a)$ and every holomorphic function $f_2$ near $\Sp(b)$, the matrices $f_1(a)$ and $f_2(b)$ are compatible.  
\end{lemma}
\begin{proof}
By assumption $[a,b_{kl}]=0$ for $k,l=1,\ldots,p$. Thus, given any holomorphic function $f_1$ near $\Sp(a)$, the coefficients $b_{kl}$ also commute with the matrix $f_1(a)$, i.e., $f_1(a)$ and $b$ are compatible. Substituting $b$ for $a$ and $f_1(a)$ for $b$ also shows that, given any holomorphic function $f_2$ near $\Sp(b)$, the matrices $f_2(b)$ and $f_1(a)$ are compatible. The proof is complete.    
\end{proof}

\begin{lemma}\label{lem:selfadjoint-selfcompatible}
 Let $g\in \GL_m^+(\cA_\theta)\cap M_m(\cA_\theta^\R)$ be self-compatible. Then $g\in \GL^+_m(\cA_\theta^\R)$. 
\end{lemma}
\begin{proof}
 Set $g^{-1}=(g^{ij})$ and $h=(g^{-1})^t$. As $g$ is selfadjoint, its inverse $g^{-1}$ is selfadjoint, and so the $(i,j)$-entry of $h$ is $h_{ij}=g^{ji}=(g^{ij})^*$. Moreover, as $g$ is self-compatible, it follows from Lemma~\ref{lem:compatible-matrices} that $g$ and $g^{-1}$ are compatible, and so $g$ and $h$ are compatible as well. In addition, by Remark~\ref{rmk:selfadjoint-Hermitian-symmetric} the fact that $g$ is selfadjoint and has selfadjoint metrics implies that $g$ is symmetric. Bearing all this mind we see that the $(i,j)$-entry of $gh$ is equal to
\begin{equation*}
 \sum_{1\leq k \leq m} g_{ik}h_{kj}=  \sum_{1\leq k \leq m} h_{kj}g_{ik} =   \sum_{1\leq k \leq m} g^{jk}g_{ki}=\delta_{ji}.  
\end{equation*}
Thus, $gh=1$, and so $h=g^{-1}$. As $h_{ij}=(g^{ij})^*$, this shows that the entries of $g^{-1}$ are selfadjoint, and hence $g\in  \GL^+_m(\cA_\theta^\R)$. The proof is complete. 
\end{proof}

We mention the following consequence of Lemma~\ref{lem:selfadjoint-selfcompatible}. 

\begin{lemma}\label{lem:Riem.self-compatible-conjugation}
 Let $g\in \GL^+_m(\cA_\theta^\R)$ and $u\in \GL_m(\cA_\theta)\cap M_m(\cA_\theta^\R)$ be compatible and self-compatible. Then $u^t gu\in  \GL^+_m(\cA_\theta^\R)$. 
\end{lemma}
\begin{proof}
Set $h=u^tgu$.  As $u$ has selfadjoint entries, its adjoint is $u^t$, and so $h=u^*gu$ is both invertible and positive, i.e., $u^tgu\in \GL^+_m(\cA_\theta)$. The $(i,j)$-entry of $h$ is $h_{ij}= \sum_{k,l} u_{ki}g_{kl}u_{lj}$. The assumption that $g$ and $u$ are compatible self-compatible matrices then ensures us that $h$ is self-compatible. Moreover, as the entries of $g$ and $u$ are selfadjoint, we also have 
\begin{equation*}
 h_{ij}^* = \big( \sum_{1\leq k,l\leq m} u_{ki}g_{kl}u_{lj}\big)^*= \sum_{1\leq k,l\leq m} u_{lj}g_{kl}u_{ki}= \sum_{1\leq k,l\leq m} u_{ki}g_{kl}u_{lj}=h_{ij}. 
\end{equation*}
Therefore, we see that $h$ is a self-compatible element of $\GL_m^+(\cA_\theta)\cap M_m(\cA_\theta^\R)$. It then follows from Lemma~\ref{lem:selfadjoint-selfcompatible} that $h\in \GL^+_m(\cA_\theta^\R)$. The proof is complete. 
\end{proof}

\begin{definition}
 A \emph{self-compatible Riemannian metric} is a Riemannian metric whose matrix is self-compatible.  
\end{definition}

\begin{remark}
As it follows from Lemma~\ref{lem:selfadjoint-selfcompatible},  a Hermitian metric on $\cX_\theta$ is a self-compatible Riemannian metric if and only if its matrix has selfadjoint entries and is self-compatible.  
\end{remark}

\begin{example}
 Conformal deformations of flat metrics as in~(\ref{eq:Riem.conf-def-Euclidean})--(\ref{eq:Riem.conf-def-flat}) are self-compatible Riemannian metrics. 
\end{example}

\begin{example}
 The \emph{functional metrics} introduced in~\cite{GK:arXiv18} are also examples of self-compatible Riemannian metrics. They are metrics of the form, 
 \begin{equation*}
 g_{ij}= g_{ij}(h),
\end{equation*}
where $h\in \cA_\theta$ and $t\rightarrow (g_{ij}(t))$ is a smooth map from a neighborhood of $\Sp(h)$ in $\R$ to $\GL_n^+(\R)$. This construction uses the observation that $\cA_\theta$ is closed under the $C^\infty$-calculus of its selfadjoint elements. 
\end{example}

Taking conformal deformations of self-compatible metrics and products of such metrics provides us with further examples of Riemannian metrics. These metrics are of the form, 
\begin{equation*}
 g = 
\begin{pmatrix}
 k_1 g^{(1)} k_1 &  & \\
 & \ddots & \\
 & & k_\ell g^{(\ell)} k_\ell
\end{pmatrix},
\end{equation*}
where $k_j \in \cA_\theta^{++}$ and $g^{(j)}\in \GL_{m_j}^+(\cA_\theta)\cap M_{m_j}(\cA_\theta)$ is self-compatible. All the previous examples of Riemannian metrics of this section are of this form. 

\section{Densities and Inner Products on Hermitian Modules}\label{sec:densities} 
In this section, after discussing the noncommutative analogues of densities on $\cA_\theta$, we explain how this enables us to define Hermitian inner products on Hermitian modules over $\cA_\theta$. 

\subsection{Densities on noncommutative tori} 
On the ordinary torus $\T^n$ a positive density (in the sense of differential geometry) is given by the integration against a positive function. The analogous notion on the noncommutative torus $\cA_\theta$ is given by elements in $\cA_\theta^{++}$ and the corresponding weights. Namely, any $\nu \in \cA_\theta^{++}$ defines a weight $\varphi_\nu:A_\theta \rightarrow \C$ by
\begin{equation}
 \varphi_\nu (x) = (2\pi)^{n} \tau (x \nu), \qquad x\in A_\theta.
 \label{eq:densities.weight} 
\end{equation}
This is a weight since $\varphi_\nu(x^*x)=\tau (x^*x\nu) = \tau[ (x\sqrt{\nu})^* (x\sqrt{\nu})]\geq 0$. In the terminology of~\cite{CM:JAMS14, CT:Baltimore11} such a weight is called a \emph{conformal weight}. 

\begin{example}
 Suppose that $\theta=0$. In this case densities are positive elements of $\cA_0=C^\infty(\T^n)$, i.e., positive $C^\infty$-functions on $\T^n$. In addition, $(2\pi)^n\tau$ agrees with the integration with respect to the Lebesgue measure of $\T^n$. Thus, given any $\nu \in C^\infty(\T^n)$, $\nu>0$, for all $f$ in $A_0=C(\T^n)$, we have
\begin{equation}
 \varphi_\nu(f)=(2\pi)^n \tau[f\nu]= \int_{\T^n} f(x)\nu(x)dx. 
\end{equation}
Therefore, we recover the usual notion of smooth densities on $C^\infty$-manifolds. 
\end{example}

\begin{lemma}\label{lem:densities.weight-tau}
 For all $x\in A_\theta^+$, we have 
 \begin{equation}
 \|\nu^{-1}\|^{-1} \tau (x) \leq (2\pi)^{-n} \varphi_\nu(x) \leq \|\nu\| \tau(x). 
 \label{eq:densities.weight-tau}
\end{equation}
\end{lemma}
\begin{proof}
As $\nu$ is invertible and positive, this is a selfadjoint element whose spectrum is contained in $[ \|\nu^{-1}\|^{-1},\|\nu\|]$, and so $ \|\nu^{-1}\|^{-1}\leq \nu \leq \|\nu\|$. Let $x\in A_\theta$. Then $ \|\nu^{-1}\|^{-1}xx^*\leq x\nu x^* \leq \|\nu\|xx^*$. As $\tau$ is a positive, the second inequality implies that we have 
\begin{equation*}
 \|\nu\|\tau(x^*x)=\|\nu\|\tau(xx^*)\geq \tau( x\nu x^*) = \tau(x^* x\nu)=(2\pi)^{-n} \varphi_\nu(x^*x) 
\end{equation*}
Likewise, we have $ \|\nu^{-1}\|^{-1}\tau(x^*x)\leq (2\pi)^{-n} \varphi_\nu(x^*x)$. This shows that the inequalities~(\ref{eq:densities.weight-tau}) hold for any positive element of $A_\theta^+$. The proof is complete. 
\end{proof}

We let $\cH_\nu$ be the Hilbert space given by the completion of $\cA_\theta$ with respect of the pre-inner product, 
\begin{equation*}
 \scal{u}{v}_\nu:= (2\pi)^{-n} \varphi_\nu(v^*u)=\tau(v^*u \nu)= \tau (u \nu v^*), \qquad u,v\in \cA_\theta. 
\end{equation*}
This inner product is a scalar multiple of the inner product obtained from the GNS construction for the state $\tilde{\varphi}_\nu(a):= \frac1{\varphi_\nu(1)}\varphi_\nu(a)= \frac1{\tau(\nu)} \tau(a\nu)$, $a\in A_\theta$. In particular, the multiplication law of $\cA_\theta$ uniquely extends to a left action of $\cA_\theta$ on $\cH_\nu$ by bounded operators. In addition, it follows from Lemma~\ref{lem:densities.weight-tau} that, for all $u \in A_\theta$, we have 
\begin{equation*}
 \|\nu^{-1}\|^{-1} \scal{u}{u} \leq \scal{u}{u}_\nu \leq \|\nu\| \scal{u}{u}.
\end{equation*}
Therefore, the completions with respect to the pre-inner products $\scal{\cdot}{\cdot}$ and $\scal{\cdot}{\cdot}_\nu$ give rise to the same locally convex space. An explicit unitary isomorphism from $\cH_\nu$ onto $\cH_\theta$ is given by the right-multiplication by $\sqrt{\nu}$, since, for all $u,v \in A_\theta$, we have
\begin{equation*}
 \scal{u}{v}_\nu = \tau\big[ (u\sqrt{\nu}) (v\sqrt{\nu})^*\big]=\scal{u\sqrt{\nu}}{v\sqrt{\nu}}. 
\end{equation*}

We also denote by $\cH_\nu^\opp$ the completion of $\cA_\theta$ with respect to the ``opposite" pre-inner product, 
\begin{equation*}
 \scal{u}{v}_\nu^\opp:= \varphi_\nu(v^*\circ u)=\tau(v^*\circ u \circ \nu)= \tau ( \nu u v^*)= \tau (v^* \nu u ), \qquad u,v\in 
 \cA_\theta. 
\end{equation*}
This inner product is a constant multiple of the inner product that arises from the GNS construction for the state $\tilde{\varphi}_\nu$ and the opposite algebra $\cA_\theta^\opp$. As above $\cH_\nu^\opp$ and its locally convex topology do not depend on $\nu$ and an explicit unitary isomorphism from $\cH_\nu^\opp$ onto $\cH_\theta$ is given by the left-multiplication by $\sqrt{\nu}$ (i.e., the right-multiplication with respect to the opposite product). 

Let $\sigma_\nu: A_\theta \rightarrow A_\theta$ be the inner automorphism given by
\begin{equation*}
 \sigma_\nu(x) = \nu^{\frac12} x \nu^{-\frac12}, \qquad x\in \cA_\theta,
\end{equation*}
Then it follows from the discussion above that, for all $u$ and $v$ in $A_\theta$, we have 
\begin{equation*}
 \scal{u}{v}_\nu^\opp = \scal{\nu^{\frac12} u}{ \nu^{\frac12} v} =  \scal{\nu^{\frac12} u\nu^{-\frac12}}{\nu^{\frac12} v \nu^{-\frac12}}_\nu = \scal{\sigma_\nu(u)}{\sigma_\nu(v)}_\nu. 
\end{equation*}
Therefore, we arrive at the following statement. 

\begin{proposition}
 The inner automorphism $\sigma_\nu$ uniquely extends to a unitary isomorphism from $\cH_\varphi^\opp$ onto $\cH_\varphi$. In particular, composing the right $*$-action of $A_\theta$ on $\cH_\nu^\opp$ with $\sigma_\nu$ gives rise to a right $*$-action of $A_\theta$ on $\cH_\nu$. 
\end{proposition}

\begin{remark}
 The unitary anti-linear involution $J_\nu(a)=\sigma_\nu(a^*)= \nu^{\frac12} a^* \nu^{-\frac12}$ and the operator $\mathbf{\Delta}(a)=\sigma_\nu^2(a)= \nu a \nu^{-1}$ are the respective Tomita involution and modular operator of the GNS construction associated with the state $\tilde{\varphi}_\nu$. In particular, we have $J_\nu a^* J_\nu=\sigma_\nu(a)$ for all $a\in A_\theta$. 
\end{remark}

\subsection{Inner products on Hermitian modules} 
Given $\nu \in \cA_\theta^{++}$, let $(\cE, \acoup{\cdot}{\cdot})$ be a Hermitian finitely generated projective (left) module over $\cA_\theta$.  Define the map $\scal{\cdot}{\cdot}_\nu:\cE \times \cE \rightarrow \C$ by 
\begin{equation}
 \scal{\xi}{\nu}_\nu = (2\pi)^{-n} \varphi_\nu \big[\acoup{\xi}{\eta}\big]= \tau \big[\acoup{\xi}{\eta} \nu\big], \qquad \xi,\eta \in \cE.
 \label{eq:densities.inner product-cEnu} 
\end{equation}

\begin{lemma}
 $\scal{\cdot}{\cdot}_\nu$ is a pre-inner product on $\cE$. 
\end{lemma}
\begin{proof}
The positivity of the weight $\varphi_\nu$ and the sesquilinearity and positivity of the Hermitian metric $\acoup{\cdot}{\cdot}$ ensure us  
that $\scal{\cdot}{\cdot}$ is a  positive sesquilinear map. It only remains to show that  $\scal{\cdot}{\cdot}$ is definite. 

Without loss of generality we may assume that $\cE=e\cA_\theta^m$ with $e\in M_m(\cA_\theta)$, $e^2=e$. We know by Corollary~\ref{cor:Positivity.equivalence-proj} that the Hermitian metric is equivalent to the restriction of the canonical Hermitian metric of $\cA^m_\theta$, and so we may further assume that the latter is actually the Hermitian metric of $\cE=e\cA_\theta^m$. In this case $\scal{\cdot}{\cdot}$ is just the restriction of the inner product of $\cH_\theta^m$, and so it is positive definite. This completes the proof. 
\end{proof}

\begin{definition}
 $\cH_\nu(\cE)$ is the Hilbert space completion of $\cE$ with respect to the pre-inner product  $\scal{\cdot}{\cdot}_\nu$. 
\end{definition}

\begin{lemma}\label{lem:inner product-Hermitian}
 $\cH_\nu(\cE)$ and its locally convex topology depend neither on $\nu$, nor on the Hermitian metric of $\cE$. 
\end{lemma}
\begin{proof}
 Let us denote by $\scal{\cdot}{\cdot}$ the pre-inner product on $\cE$ associated with $\nu=1$, i.e., $\scal{\xi}{\eta}=\tau[\acoup{\xi}{\eta}]= \tau[\acoup{\xi}{\eta}]$, $\xi,\eta \in \cE$. Let $\acoup{\cdot}{\cdot}_1$ be another Hermitian metric on $\cE$ and denote by $\scal{\cdot}{\cdot}_{1,\nu}$ the pre-inner product~(\ref{eq:densities.inner product-cEnu}) associated with $\nu$ and  $\acoup{\cdot}{\cdot}_1$. Thanks to Corollary~\ref{cor:Positivity.equivalence-proj} there are constants $c_1>0$ and $c_2>0$ such that 
 \begin{equation}
 c_1\acoup{\xi}{\xi} \leq \acoup{\xi}{\xi}_1 \leq c_2\acoup{\xi}{\xi} \qquad \text{for all $\xi \in \cE$}. 
 \label{eq:densities.acoup-c1c2}
\end{equation}

Let $\xi\in \cE$. Combining~(\ref{eq:densities.acoup-c1c2}) with the positivity of $\varphi_\nu$ gives
\begin{equation*}
 c_1 \varphi_\nu [\acoup{\xi}{\xi} ] \leq \varphi_\nu [\acoup{\xi}{\xi}_1 ]  \leq 
c_2\varphi_\nu [\acoup{\xi}{\xi} ].
\end{equation*}
Combining this with Lemma~\ref{lem:densities.weight-tau} we get
\begin{equation*}
 \|\nu^{-1}\|^{-1} c_1\tau \big[\acoup{\xi}{\xi} \big ] \leq (2\pi)^{-n} \varphi_\nu \big[\acoup{\xi}{\xi}_1 \big] \leq  \|\nu\| c_2 \tau \big[\acoup{\xi}{\xi} \big ]. 
\end{equation*}
 As $\tau[\acoup{\xi}{\xi}]=\scal{\xi}{\xi}$ and $\varphi_\nu [\acoup{\xi}{\xi}_1 ] =\scal{\xi}{\xi}_{1,\nu}$, we see that there are $c_1'>0$ and $c_2'>0$ such that
 \begin{equation*}
c_1' \scal{\xi}{\xi} \leq \scal{\xi}{\xi}_{1,\nu} \leq c_2' \scal{\xi}{\xi} \qquad \text{for all $\xi\in \cE$}. 
\end{equation*}
It then follows that the completions of $\cE$ with respect to the pre-inner products $\scal{\cdot}{\cdot}$ and $\scal{\cdot}{\cdot}_{1,\nu}$ give rise to the same locally convex topological vector space. This proves the result.
 \end{proof}

\section{Determinant}\label{sec:determinant}
The aim of this section is to introduce a notion of determinant on $\GL_m^+(\cA_\theta)$. This will be used in the next section to define the volume of $\cA_\theta$ with respect to an arbitrary Riemannian metric.

It will be convenient to work more generally with a unital involutive subalgebra $\cA$ of a unital $C^*$-algebra $A$ such that $\cA$ and all the matrix algebras $M_m(\cA)$, $m\geq 2$, are closed under holomorphic functional calculus. For instance, we may take $\cA=\cA_\theta$, or more generally we may take $\cA$ to be any pre-$C^*$-algebra.  

To define the determinant we shall essentially follow the approach of Fuglede-Kadison~\cite{FK:AM52}. Recall if $h\in \GL_m^+(\C)$, then its determinant is given by 
\begin{equation}
 \det(h) = \exp \big[ \Tr (\log h) \big]. 
 \label{eq:det.C} 
\end{equation}
As in~\cite{FK:AM52}, we shall use this formula to define the determinant for matrices in $\GL_m^+(\cA)$. 

The exponential map $\exp: A\rightarrow A$ is defined by means of the holomorphic functional calculus or, equivalently, by using the usual power series expansion, 
\begin{equation*}
 \exp(x) = \sum_{k=0}^\infty \frac{1}{k!} x^k, \qquad x\in A.
\end{equation*}
In particular, by using this power expansion and the binomial formula it can be shown that
\begin{equation}
 \exp(x+y)=\exp(x)\exp(y) \qquad \text{whenever}\ [x,y]=0. 
 \label{eq:det.exp-sum}
\end{equation}
In addition, as $\cA$ is closed under holomorphic functional calculus, $\log(x)\in \cA$ when $x\in \cA$. Moreover, if $x\in \cA$ is selfadjoint, then $\exp(x)$ is selfadjoint and has positive spectrum, and so it belongs to $A^{++}$. Thus, if we denote by $\cA^\R$ the real subspace of selfadjoint elements of $\cA$, then
\begin{equation*}
 \exp(x) \in \cA^{++} \qquad \text{for all $x\in \cA^\R$}.  
\end{equation*}

Let $h\in \GL^+_m(\cA)$. Then $\log h$ is a normal element of $A$ with real spectrum, and hence it is selfadjoint. Moreover, as $M_m(\cA)$ is closed under holomorphic functional calculus, we see that $\log h\in M_m(\cA)$. Thus, $\log h$ is a selfadjoint element of $M_m(\cA)$. In addition, we mention the following additivity property of the logarithm on $\GL^+_m(\cA)$. 

\begin{lemma}\label{lem:det.logab} 
 Let $h,h' \in \GL_m^+(\cA)$ be such that $[h,h']=0$. Then $hh' \in \GL_m^+(\cA)$ and we have 
\begin{equation}
        \log (hh')=\log (h) + \log (h').
        \label{eq:det.logab} 
\end{equation}
\end{lemma}
\begin{proof}
It is well known that as $h$ and $h'$ are positive and commute, their product $hh'$ is positive, since in this case $hh'=h^{\frac12}h'h^{\frac12}= (h^{\frac12}(h')^{\frac12})^2$. As $hh'$ is invertible, we then see that $hh'\in \GL_m^+(\cA)$. 

Bearing this in mind, let $B$ be the unital Banach algebra generated by $h$ and $h'$. As $h$ and $h'$ are selfadjoint this is actually a $C^*$-algebra. Moreover, as $[h,h']=0$ we see that $B$ is a commutative. In particular, the identity map is a trace. The formula~(\ref{eq:det.logab}) then follows from~\cite[Appendix, Lemma~5]{Co:AIM81} and~\cite[Lemme~1-(a)]{dHS:AIF84} (see also~\cite{FK:AM52}). The proof is complete. 
\end{proof}

As we have just seen,  if $h \in \GL^+_m(\cA)$, then  $\log h$ is a selfadjoint element of $M_m(\cA)$. 
This implies that the trace  $\Tr[\log  h]$ is a selfadjoint element of $\cA$, and so by taking the exponential we obtain an element of $\cA^{++}$. 

\begin{definition}
 The \emph{determinant} $\det: \GL^+_m(\cA)\rightarrow \cA^{++}$ is defined by
 \begin{equation*}
 \det(h) = \exp \big[ \Tr (\log h) \big], \qquad h \in \GL^+_m(\cA). 
\end{equation*}
\end{definition}

\begin{remark}
 The value of $\det h$ depends only on the holomorphic functional calculus closure of the algebra generated by the entries of $h$.  
\end{remark}

\begin{proposition}\label{prop:det.properties}
The following holds. 
\begin{enumerate}
 \item We have 
           \begin{equation}
              \det (kI_m) =k^m \qquad \forall k\in \cA^{++},
                  \label{eq:det.kIn}
            \end{equation}
            
\item Let $h\in \GL^+_m(\cA)$. Then
          \begin{gather*}
                   \det (t h)=t^m \det (h) \qquad \forall t>0,\\
                    \det (h^s)= (\det h)^s\qquad \forall s \in \R. 
          \end{gather*}
\end{enumerate}
\end{proposition}
\begin{proof}
 Let $k\in \cA^{++}$. As $\Tr[ \log(k I_m)]=\Tr [ (\log k)I_m]=m \log k$, we have
\begin{equation*}
 \det (k I_m)= \exp\big( m \log k\big)= \big( \exp(\log k)\big)^m=k^m. 
\end{equation*}
 
Let $h\in \GL^+_m(\cA)$ and $t>0$. As $\log (tz)=\log t + \log z$, we see that $\log(th)=\log t + \log h$, and so we have 
  \begin{equation*}
 \det (t h)=\exp \big( \Tr[ \log(th)]\big) = \exp \big( m \log t + \Tr[\log h]\big).  
\end{equation*}
As $\exp(m \log t+ z)=t^m \exp(z)$ for all $z\in \C$, we also have 
 \begin{equation*}
\exp \big( m \log t + \Tr[\log h]\big)=t^m \exp \big( \Tr[\log h]\big)=t^m \det (h).  
\end{equation*}
It then follows that $ \det (t h)=t^m \det (h)$. 

Let $s\in \R$. As $\log (z^s)=s \log z$ we have $\log (h^s)=s \log h$. In addition, as $\exp (sz)=(\exp(z))^s$ for all $z\in \C$,  we have $\exp (sx)=(\exp x)^s$ for all $x \in \cA$. Thus, 
\begin{equation*}
 \det (h^s)= \exp\big( \Tr[ \log (h^s)]\big) = \exp \big( s \Tr[\log h]\big)= \left[ \exp \big(  \Tr[\log h]\big) \right]^s= (\det h)^s. 
\end{equation*}
The proof is complete. 
\end{proof}

\begin{proposition}\label{prop:det.commutative}
 Suppose that $A$ is commutative. Then, for all $h=(h_{ij})\in \GL_m^+(A)$, we have
 \begin{equation}
 \det (h)= \sum_{\sigma \in \fS_m} \varepsilon(\sigma) h_{1\sigma(1)} \cdots h_{m \sigma(m)},
 \label{eq:det.Leibniz}
\end{equation}
 where $ \fS_m$ is the permutation group of $\{1, \ldots ,m\}$. 
\end{proposition}
\begin{proof}
 As $A$ is a commutative unital $C^*$-algebra, we may assume that $A=C(X)$ for some compact Hausdorff topological space. In that case  $M_m(C(X))= C(X,M_m(\C))$ and $\GL^+_m(C(X))=C(X, \GL^+_m(\C))$. If $h(x)=(h_{ij}(x))$ is in $C(X, \GL^+_m(\C))$, then $(\log h)(x)=\log[h(x)]$. Thus, by using~(\ref{eq:det.C}) we get
 \begin{equation*}
 (\det h)(x) = \exp\left( \Tr \big [ \log (h(x))\big]\right)= \det[h(x)]= \sum_{\sigma \in \fS_m} \varepsilon(\sigma) h_{1\sigma(1)}(x) \cdots h_{m \sigma(m)}(x). 
\end{equation*}
This proves the result. 
\end{proof}

We define notions of compatibility and self-compatibility of matrices with entries in $\cA$ in the same was way as in Definition~\ref{def:Riem.self-compatible}. All the properties of these matrices proved in Section~\ref{sec:Riem} hold \emph{verbatim} in the setting of matrices with entries in $\cA$. 

\begin{corollary}\label{cor:det.self-comp}
 Let $h\in \GL_m^+(\cA)$ be self-compatible. Then $\det (h)$ is given by~(\ref{eq:det.Leibniz}). 
\end{corollary}
\begin{proof}
 Let $\cB$ be the unital algebra generated by the entries $h_{ij}$ of $h$. The fact that $h$ is selfadjoint means that $h_{ij}^*=h_{ji}$. Therefore, the sub-algebra $\cB$ is involutive and its closure $B=\overline{\cB}$ is a commutative unital $C^*$-algebra. It then follows from Proposition~\ref{prop:det.commutative} that $\det (h)$ is given by~(\ref{eq:det.Leibniz}). The proof is complete. 
\end{proof}

\begin{lemma}\label{lem:det.det-compatible}
 Let $h\in \GL_m^+(\cA)$ and $h'\in \GL_{m'}^+(\cA)$ be compatible. Then $ \big[ \det(h), \det(h')\big]=0$. 
 \end{lemma}
\begin{proof}
As $h$ and $h'$ are compatible, we know by Lemma~\ref{lem:compatible-matrices}  that $\log (h)$ and $\log (h')$ are compatible. This means that each entry of $\log(h)$ commutes with every entry of $\log(h')$. This implies that $[\Tr [\log (h)],\Tr[\log (h')]]=0$. It then follows from~(\ref{eq:det.exp-sum}) that $\det(h) =\exp(\Tr[\log (h)])$ and $\det(h') =\exp(\Tr[\log (h')])$ commute with each other. The lemma is proved. 
\end{proof}

\begin{proposition}\label{prop:det.diagonal}
 Let $h\in \GL_m^+(\cA)$ be block-diagonal, i.e.,  
\begin{equation*}
 h= 
\begin{pmatrix}
h^{(1)} & & 0 \\
  & \ddots & \\
0  & & h^{(\ell)}  
\end{pmatrix}, \qquad h^{(j)}\in \GL_{m_j}^+(\cA). 
 \end{equation*}
 Assume further that $h^{(i)}$ and $h^{(j)}$ are compatible for $i\neq j$. Then
\begin{equation*}
 \det (h) = \det(h^{(1)}) \cdots \det(h^{(\ell)}).  
\end{equation*}
\end{proposition}
\begin{proof}
 It is enough to prove the result when $\ell =2$, since once this case is proved the general result follows by induction. Assume that $\ell=2$. Then 
\begin{equation*}
 \Tr[ \log h] = \Tr 
\begin{pmatrix}
 \log h^{(1)} & 0 \\
 0 &  \log h^{(2)} 
\end{pmatrix} = \Tr\big[\log h^{(1)}\big] +  \Tr\big[\log h^{(2)}\big]. 
\end{equation*}
Therefore, we have $\det (h) = \exp ( \Tr[ \log h] )= \exp( \Tr[\log h^{(1)}] +  \Tr[\log h^{(2)}])$. As $h^{(1)}$ and $h^{(2)}$ are compatible, we know from the proof of Lemma~\ref{lem:det.det-compatible} that $\Tr[\log h^{(1)}]$ and $\Tr[\log h^{(2)}]$ commute with each other. Thus, by using~(\ref{eq:det.exp-sum}) we get
\begin{equation}
 \det(h) = \exp\big( \Tr[\log h^{(1)}]\big) \exp\big( \Tr[\log h^{(2)}]\big)= \det(h^{(1)}) \det(h^{(2)}). 
 \label{eq:det.block-diag}
\end{equation}
This completes the proof. 
\end{proof}

Because $\Tr$ is not a trace on $M_m(\cA)$ unless $\cA$ is commutative, we cannot expect the determinant on $\GL_m^+(\cA)$ to be multiplicative in general (compare~\cite{FK:AM52}). Nevertheless, we have the following results. 

\begin{proposition}\label{prop:det.product} 
 Suppose that $h,h'\in \GL^+_m(\cA)$ are compatible and commute with each other. Then $hh'\in \GL_m^+(\cA)$ and we have
 \begin{equation*}
 \det (hh')= \det(h) \det(h')=\det(h')\det(h). 
\end{equation*}
\end{proposition}
\begin{proof}
 As $h$ and $h'$ are compatible, we know by Lemma~\ref{lem:det.det-compatible} that $\det(h) \det(h')=\det(h')\det(h)$. Moreover, as $[h,h']=0$ Lemma~\ref{lem:det.logab} ensures us that $hh' \in \GL_m^+(\cA)$ and $\log (hh')=\log (h) + \log(h')$. Thus, $\det (hh')= \exp(\Tr [\log h]+ \Tr [\log h'])$. In addition, as $h$ and $h'$ are compatible in the same way as in~(\ref{eq:det.block-diag}) we get
 \begin{equation*}
 \det (hh') =  \exp\big( \Tr[\log h]\big) \exp\big( \Tr[\log h']\big)= \det(h) \det(h'). 
\end{equation*}
The proof is complete.  
\end{proof}

\begin{proposition}\label{prop:det.conjugaison}
 Let $h \in \GL^+_m(\cA)$ and $u \in \GL_m(\cA)$ be compatible and self-compatible. Assume further that $u$ and $u^*$ are compatible with each other. Then
 \begin{equation*}
 \det (u^* h u)= \det (u^*u)\det(h). 
\end{equation*}
\end{proposition}
\begin{proof}
 Let $\cB$ be the unital algebra generated by the entries of $h$, $u$ and $u^*$. We observe that these entries form an involutive generator set, since $h$ is selfadjoint and the entries of $u^*$ (resp., $u$) are the adjoints of the entries of $u$ (resp., $u^*$). By assumption $h$ and $u$ are self-compatible. Note that the self-compatibility of $u$ implies that $u^*$ is  self-compatible as well. By assumption $u$ is compatible with $h$ and $u^*$. As $h$ is selfadjoint we also see that $h$ and $u^*$ are compatible. It then follows that the entries of $h$, $u$ and $u^*$ form a commutative involutive generator set of $\cB$. Therefore, the unital subalgebra $\cB$ is involutive and commutative, and so its closure $B:=\overline{\cB}$ is a commutative $C^*$-algebra. We then know by Proposition~\ref{prop:det.commutative} that the restriction to $\GL_m^+(B)$ of the determinant agrees with the usual determinant on a commutative ring. In particular, it enjoys all the standard properties of the determinant on commutative rings, including multiplicativity. Thus, 
\begin{equation*}
 \det (u^* h u)= \det (u^*)\det(h)\det(u)=\det(u^*)\det(u)\det(h)=\det(u^*u)\det(h). 
\end{equation*}
This proves the result. 
\end{proof}

As an immediate consequence of Proposition~\ref{prop:det.conjugaison} we get the following unitary invariance result. 

\begin{corollary}
 Let $h \in \GL^+_m(\cA)$ and $u \in \GL_m(\cA)$ be compatible and self-compatible. Assume further that $u$ is unitary, i.e., $u^*u=uu^*=1$. Then
 \begin{equation*}
 \det (u^* h u)= \det(h). 
\end{equation*}
\end{corollary}

\section{Riemannian Densities and Riemannian Volume}\label{sec:Riem-volume} 
In this section, we introduce a notion of Riemannian densities and Riemann volume on $\cA_\theta$ for arbitrary Riemannian metrics. 

Recall that if $g=(g_{ij}(x))$ is a Riemannian metric on an ordinary manifold $M^n$, then its Riemannian density is given by the integration against $\sqrt{\det(g(x))}$. When $M$ is compact, its volume with respect to $g$ is defined by 
\begin{equation*}
 \Vol_g(M):= \int_M \sqrt{\det(g(x))} d^nx. 
\end{equation*}

Let $g \in \GL^+_n(\cA_\theta^\R)$ be a Riemannian metric on $\cA_\theta$. As defined in the previous section, the determinant $\det(g)= \exp(\Tr [\log g])$ is an element of $\cA_\theta^{++}$. As explained in Section~\ref{sec:densities}, elements $\nu \in \cA_\theta^{++}$ and their associated weights $\varphi_\nu$ given by~(\ref{eq:densities.weight})  are the analogues of (smooth) densities on ordinary manifolds. This leads us to the following. 

\begin{definition}
 The \emph{Riemannian density} of $g$ is 
\begin{equation*}
 \nu(g):=\sqrt{\det(g)} = \exp\left(\frac12 \Tr [\log g]\right)\in \cA_\theta^{++}. 
\end{equation*}
The \emph{Riemannian weight} of $g$ is the weight $\varphi_g:=\varphi_{\nu(g)}$, i.e., 
\begin{equation*}
 \varphi_{g}(u) = (2\pi)^{n}\tau\left[ u \nu(g)\right], \qquad u \in \cA_\theta. 
\end{equation*}
\end{definition}

This allows us to define the Riemannian volume as follows. 

\begin{definition}
 The \emph{Riemannian volume} of $\cA_\theta$ with respect to $g$ is 
\begin{equation*}
 \Vol_g(\cA_\theta):= \varphi_g(1)= (2\pi)^{n} \tau\left[  \nu(g)\right]. 
\end{equation*}
\end{definition}

\begin{example}
 If $g=\delta_{ij}$ is the flat Euclidean metric, then $\nu(1)=1$, and so we have \[ \Vol_g(\cA_\theta)=(2\pi)^n\tau[1]=(2\pi)^n=|\T^n|.\] 
\end{example}

\begin{example}
 If $g=k^2 \delta_{ij}$, $k \in \cA_\theta^{++}$, is a conformal deformation of the Euclidean metric, then by~(\ref{eq:det.kIn}) we have
\begin{equation}
 \nu(g)= \sqrt{\det(k^2 I_n)} = \sqrt{k^{2n}}=k^n. 
 \label{eq:vol.k2In}
\end{equation}
\end{example}

\begin{proposition}\label{prop:vol.product}
Let $g$ be a  product metric of the form,  
\begin{equation*}
 g= \begin{pmatrix}
g^{(1)} & & 0 \\
  & \ddots & \\
0  & & g^{(\ell)}  
\end{pmatrix}, \qquad g^{(j)}\in \GL_{m_j}^+(\cA_\theta^\R),  
 \end{equation*}
 where $g^{(1)}, \ldots, g^{(\ell)}$ are mutually compatible. Then, we have 
\begin{equation*}
 \nu(g)= \nu\big( g^{(1)} \big) \cdots \nu\big( g^{(\ell)} \big). 
\end{equation*}
\end{proposition}
\begin{proof}
 We know by Proposition~\ref{prop:det.diagonal} that $\det(g)= \det(g^{(1)}) \cdots  \det(g^{(\ell)})$. As  $ \det(g^{(1)})$, ... , $\det(g^{(\ell)})$ mutually commute with each other (\emph{cf}. Lemma~\ref{lem:det.det-compatible}), we get 
\begin{equation}
 \nu(g)= \sqrt{ \det(g^{(1)}) \cdots  \det(g^{(\ell)})}=  \sqrt{\det(g^{(1)})} \cdots  \sqrt{\det(g^{(\ell)})}= \nu\big( g^{(1)} \big) \cdots \nu\big( g^{(\ell)} \big).
\label{eq:vol.sqrt-product} 
\end{equation}
The result is proved. 
\end{proof}

\begin{example}
 Let $g$ be a product of conformal deformations of the Euclidean metric, i.e., 
\begin{equation*}
 g = 
\begin{pmatrix}
 k_1^2 I_{m_1}  & 0 \\
  0 & k_2^2 I_{m_2}
\end{pmatrix}, \qquad k_j \in \cA_\theta^{++}. 
\end{equation*}
If $[k_1,k_2]=0$, then by using Proposition~\ref{prop:vol.product} and~(\ref{eq:vol.k2In}) we get
\begin{equation*}
 \nu(g) = \nu(k_1^2 I_{m_1}) \nu(k_2^2 I_{m_2}) = k_1^{m_1}  k_2^{m_2}. 
\end{equation*}
\end{example}
 
\begin{proposition}\label{prop:volume.conformal-transformation}
 Let $g \in \GL_n^+(\cA_\theta^\R)$ and $k \in \cA_\theta^{++}$ be such that $[k,g]=0$. Then, we have 
\begin{equation*}
 \nu(k^2 g)= k^n \nu(g)=\nu(g)k^n.  
\end{equation*}
\end{proposition}
\begin{proof}
As $[k,g]=0$ the matrices $k^2I_n$ and $g$ satisfy the assumptions of Proposition~\ref{prop:det.product}. Combining this with~(\ref{eq:det.kIn}) we get 
\begin{equation*}
 \det(k^2g) = \det(k^2 I_n) \det(g)= k^{2n} \det(g).
\end{equation*}
 Note that $[k,\det(g)]=0$ (\emph{cf}.~Lemma~\ref{lem:det.det-compatible}). Therefore, in the same way as in~(\ref{eq:vol.sqrt-product}) we have
\begin{equation*}
 \nu(k^2g)= \sqrt{k^{2n}\det(g)}=k^n \sqrt{\det(g)} = k^n \nu(g). 
\end{equation*}
The proof is complete. 
\end{proof}

\begin{proposition}
 Let $g\in \GL_n^+(\cA_\theta^\R)$ and $u\in \GL_n(\cA_\theta^\R)$ be compatible and self-compatible. Then,
\begin{equation*}
 \nu(u^tg u)= \nu(u^tu)\nu(g)=\nu(g)\nu(u^tu). 
\end{equation*}
If we further assume $u$ to be orthogonal, i.e., $u^tu=uu^t=1$, then we have
\begin{equation*}
 \nu(u^t g u)=\nu(g) \qquad \text{and} \qquad \Vol_{u^t gu}(\cA_\theta)= \Vol_{g}(\cA_\theta). 
\end{equation*}
\end{proposition}
\begin{proof}
 We know by Lemma~\ref{lem:Riem.self-compatible-conjugation} that $u^tu$ and $u^t gu$ are both elements of $\GL_n^+(\cA_\theta^\R)$. Moreover, as $g$ is compatible with $u$, it is also compatible with $u^tu$, and so Lemma~\ref{lem:det.det-compatible} ensures us that $[\det(u^t u),\det(g)]=0$. In addition, as $u^t=u^*$ the assumptions of Proposition~\ref{prop:det.conjugaison} are satisfied, and so we have $\det(u^t g u)=\det(u^tu) \det(g)$, and hence $\nu(u^tg u)= \nu(u^tu)\nu(g)=\nu(g)\nu(u^tu)$. 
 
Assume further that $u$ is orthogonal. Then $u^tu=1$, and so $\nu(u^tu)=\sqrt{\det(1)}=\sqrt{1}=1$. Thus $\nu(u^t g u)=\nu(g)$, and hence $\Vol_{u^t gu}(\cA_\theta)= \Vol_{g}(\cA_\theta)$. The proof is complete. 
\end{proof}

\begin{remark}
 When $(\cE, \acoup{\cdot}{\cdot})$ is a Hermitian module over $\cA_\theta$, we shall simply denote by $\cH_g(\cE)$ the Hilbert space $\cH_{\nu(g)}(\cE)$. 
\end{remark}

\section{Differential $1$-Forms and Divergence Operator}\label{sec:1-forms}
In this section, we describe the bimodule of differential 1-forms on the noncommutative torus $\cA_\theta$ and introduce a divergence operator on this bimodule. 

As we think of the module $\cX_\theta$ as the module of vector fields on $\cA_\theta$, we define $1$-forms as elements of the dual of $\cX_\theta$. 

\begin{definition}
$\Omega_\theta^1$ is the dual module $\Hom_{\cA_\theta}(\cX_\theta, \cA_\theta)$. 
\end{definition}

By definition $\Omega^1_\theta$ is a right module over $\cA_\theta$. This is a free module; a basis is provided by the dual basis $(\theta^1, \ldots, \theta^n)$ of the basis 
$(\dl_1,\ldots, \dl_n)$ of $\cX_\theta$. That is, 
\begin{equation*}
 \acou{\theta^i}{\dl_j}=\delta_{ij}, \qquad i,j=1,\ldots,n.
\end{equation*}
Thus, any $\omega \in \Omega^1_\theta$ has a unique decomposition $\omega = \sum \theta^i \omega_i$ with $\omega_i=\acou{\omega}{\dl_i}$. We also have a left-action of $\cA_\theta$ on $\Omega^1_\theta$ given by 
\begin{equation*}
 a \omega = \sum_{1\leq i \leq n} \theta^i (a\omega_i), \qquad a \in \cA_\theta, \  \omega = \sum \theta^i\omega_i \in \Omega_\theta^1. 
\end{equation*}
This left action is compatible with the right action of $\cA_\theta$, and so this turns $\Omega^\theta_1$ into an $\cA_\theta$-bimodule. Note also that with respect to this left action we have $a \theta^i = \theta^i a$, $a \in \cA_\theta$. Thus, for any $\omega \in \Omega^1_\theta$, we have 
\begin{equation*}
 \omega= \sum_{1\leq i \leq n} \theta^i\omega_i= \sum_{1\leq i \leq n} \omega_i\theta^i, \qquad \omega_i=\acou{\omega}{\dl_i}. 
\end{equation*}

\begin{definition}
 The \emph{differential} of any $u\in \cA_\theta$ is 
\begin{equation*}
 du :=\sum_{1\leq i \leq n} \theta^i \dl_i(u)\in \Omega^1_\theta. 
\end{equation*}
\end{definition}

\begin{proposition}\label{prop:1-forms.properties-differential}
The following holds. 
\begin{enumerate}
 \item Leibniz Rule: for all $u,v\in \cA_\theta$, we have
 \begin{equation*}
 d(uv)= (du)v + u(dv). 
 \end{equation*} 

 \item  Let $u\in \cA_\theta$. Then $du =0$ if and only if $u\in \C$. That is, $\ker d =\C$. 
\end{enumerate}
\end{proposition}
\begin{proof}
 Let $u,v\in \cA_\theta$. We have $d(uv)= \sum \theta^i \dl_i(uv)$. As $\dl_1, \ldots, \dl_n$ are derivations of the algebra $\cA_\theta$, we have
 \begin{equation*}
 d(uv) = \sum_{1\leq i \leq n} \theta^i (\dl_i u) v +  \sum_{1\leq i \leq n} \theta^i u (\dl_i v) = (du)v + u(dv).
\end{equation*}

It remains to show that $\ker d=\C$. We have $d(1)=0$, and so $\C\subset \ker d$. Conversely, let $u\in \ker d$,  i.e., $\sum_j \theta^j \dl_j(u)=0$. As $\theta^1,\ldots, \theta^n$ form a basis of the right module $\Omega^1_\theta$, we see that $\dl_j(u)=0$ for $j=1,\ldots,n$. Set $u=\sum_{k} u_k U^k$. 
Then $0=\dl_j(u)= \sum_k ik_j u_k U^k$, and so $u_k=0$ whenever $k_j\neq 0$. As this holds for $j=1,\ldots,n$, we deduce that $u_k=0$ whenever $k\neq 0$. That is, $u=u_0\in \C$. Therefore, we see that $\ker d =\C$. The proof is complete.  
\end{proof}

In the same way as on $\cX_\theta$, any $h\in \GL^+_n(\cA_\theta)$ defines a Hermitian metric $\acoup{\cdot}{\cdot}_h$ on $\Omega^1_\theta$. Namely, 
given $\omega = \sum \theta^i  \omega_i$ and $\zeta = \sum  \theta^i\zeta_i$ in $\Omega_\theta^1$, we have 
\begin{equation*}
 \acoup{\omega}{\zeta}_h= \sum_{1\leq i,j \leq n}  \omega_i \circ h_{ij}  \circ \zeta_j^* = \sum_{1\leq i,j \leq n} \zeta_j^*   h_{ij} \omega_i = 
  \sum_{1\leq i,j \leq n} \zeta_i^* h_{ij}^* \omega_j. 
\end{equation*}
where we have used the fact that $h_{ij}=h_{ji}^*$ to get the last equality.  Conversely, any Hermitian metric on $\Omega_\theta^1$ is of this form. 

Let $h \in \GL^+(\cA_\theta)$. The nondegeneracy of the Hermitian metric $\acoup{\cdot}{\cdot}_h$ on $\cX_\theta$ means that, for every 1-form $\omega \in \Omega_\theta^1$, there is a unique vector field $X_\omega^h \in \cX_\theta$ such that 
\begin{equation*}
 \acou{\omega}{Y}= \acoup{Y}{X_\omega^h} _h \qquad \text{for all $Y\in \cX_\theta$}. 
\end{equation*}
The dual Hermitian metric $\acoup{\cdot}{\cdot}_h'$ on $\Omega_\theta^1$ is then given by
\begin{equation}
 \acoup{\omega}{\zeta}_h'=\acoup{X_\zeta^h}{X_\omega^h}_h, \qquad \omega, \zeta \in \Omega_\theta^1.
 \label{eq:Riemannian.dual-metric} 
\end{equation}

Set $h^{-1}=(h^{ij})$. Note that $h^{-1}\in \GL_n^+(\cA_\theta)$.  

\begin{lemma}
 For all $\omega = \sum_i \theta^i \omega_i $ and $\zeta = \sum_i \theta^i  \zeta_i $ in $\Omega_\theta^1$, we have
\begin{gather}
X_\omega^h = \sum_{1\leq i,j\leq n} \omega_j^* h^{ji} \dl_i,
\label{eq:Riemannian.Xomg} \\ 
 \acoup{\omega}{\zeta}'_h =  \acoup{\omega}{\zeta}_{(h^{-1})^t}= \sum_{1\leq i,j \leq n} \zeta_i^* h^{ij} \omega_j.
 \label{eq:Riemannian.dual-metric-g}  
\end{gather}
\end{lemma}
\begin{proof}
 Let us denote by $\Phi$ the anti-linear isomorphism $\cX_\theta \ni X \rightarrow \acoup{\cdot}{X}_h \in \Omega_\theta^1$ defined by the Hermitian metric $\acoup{\cdot}{\cdot}_h$. Let $\omega = \sum_i \theta^i \omega_i\in \Omega^1_\theta$. By definition, we have 
 \begin{equation}
 X_\omega^h = \Phi^{-1}(\omega)= \sum_{1\leq i \leq n} \Phi^{-1} \big( \theta^i \omega_i\big) = \sum_{1\leq i \leq n}\omega_i^* \Phi^{-1} \big( \theta^i\big). 
 \label{eq:Riemannian.Xomg-Phi}
\end{equation}
For $i=1,\ldots, n$, we have 
\begin{equation*}
\Phi(\dl_i)= \sum_{1\leq j \leq n}  \theta^j \acou{\Phi(\dl_i)}{\dl_j} =  \sum_{1\leq j \leq n}  \theta^j \acoup{\dl_j}{\dl_i}_h=  \sum_{1\leq j \leq n}  \theta^j h_{ji}. 
\end{equation*}
It then follows that $\Phi^{-1}(\theta^i)= \sum_{j=1}^n h^{ij}\dl_j$. Indeed, as $h$ is selfadjoint,  we have
\begin{equation*}
 \Phi\big(  \sum_{1\leq j \leq n} h^{ij}\dl_j \big) = \sum_{1\leq j \leq n}  \Phi (  \dl_j)(h^{ij})^*= \sum_{1\leq j,k \leq n} \theta^k h_{kj}h^{ji} =\theta^i. 
\end{equation*}
By using~(\ref{eq:Riemannian.Xomg-Phi}) we then get 
\begin{equation*}
 X_\omega^h = \sum_{1\leq i \leq n}\omega_i^* \Phi^{-1} \big( \theta^i\big) = \sum_{1\leq i, j \leq n} \omega_i^* h^{ij}\dl_j =  \sum_{1\leq i,j\leq n} \omega_j^* h^{ji} \dl_i. 
\end{equation*}
This gives~(\ref{eq:Riemannian.Xomg}). 

Let $\zeta = \sum \theta^i\zeta_i\in \Omega_\theta^1$. In view of~(\ref{eq:Riemannian.dual-metric}) we have 
\begin{equation*}
 \acoup{\omega}{\zeta}_h'=\acoup{X_\zeta^h}{X_\omega^h}_h= \sum_{1\leq i,j\leq n} (X^h_\zeta)_i h_{ij} (X^h_\omega)_j^*.
\end{equation*}
Therefore, by using~(\ref{eq:Riemannian.Xomg}) we see that $\acoup{\omega}{\zeta}_h'$ is equal to
\begin{equation*}
 \sum_{1\leq i,j\leq n} \sum_{1\leq k,l \leq n}  \big(\zeta_k^* h^{ki}\big) h_{ij} 
 \big(\omega_l^* h^{lj}\big)^* =  \sum_{1\leq k,l \leq n} \sum_{1\leq i,j\leq n} 
 \zeta_k^* h^{ki} h_{ij} (h^{lj})^* \omega_l= \sum_{1\leq k,l \leq n}  \zeta_k^* h^{kl} \omega_l. 
\end{equation*}
This proves~(\ref{eq:Riemannian.dual-metric-g}). The proof is complete. 
 \end{proof}

From now on we let $\nu \in \cA_\theta^{++}$. As mentioned above the lack of unitarity of the right action of $\cA_\theta$ on $\cH_\nu^0$ is corrected by means of the inner automorphism $\sigma_\nu(x)=\nu^{\frac12} x \nu^{-\frac12}$. This inner automorphism makes sense on any bimodule over $\cA_\theta$, including $\Omega^1_\theta$ and $M_n(\cA_\theta)$. For instance, for any $1$-form $\omega = \sum \theta^i\omega_i$, we have
\begin{equation}
 \sigma_\nu (\omega) = \nu^{\frac12} \omega \nu^{-\frac12} = \sum \theta^i  \nu^{\frac12} \omega_i \nu^{-\frac12}.
 \label{eq:sigma-nu-forms} 
\end{equation}

In what follows, we let $\scal{\cdot}{\cdot}_{h,\nu}^\opp$ be the pre-inner product on $\Omega_\theta^1$ defined by 
\begin{equation*}
 \scal{\omega}{\zeta}_{h,\nu}^\opp = \varphi_\nu \big[ \acoup{\sigma_\nu(\omega)}{\sigma_\nu(\zeta)}_{(h^{-1})^t}\!\!\big], \qquad \omega, \zeta \in \Omega^1_\theta. 
\end{equation*}
 
\begin{lemma}
 For all $\omega = \sum \theta^i \omega_i$ and $\zeta = \sum \theta^i \zeta_i$ in $\Omega^1_\theta$, we have 
 \begin{equation}
 \scal{\omega}{\zeta}_{h,\nu}^\opp =  \sum_{1\leq i,j \leq n} \tau  \big[\zeta_i^* \nu^{\frac12} h^{ij} \nu^{\frac12} \omega_j \big]. 
 \label{eq:div.inner product-forms-nu}
\end{equation}
\end{lemma}
\begin{proof}
 By definition  $ \scal{\omega}{\zeta}_{h,\nu}^\opp = \varphi_\nu [ \acoup{\sigma_\nu(\omega)}{\sigma_\nu(\zeta)}_{(h^{-1})^t}] = \tau [ \acoup{\sigma_\nu(\omega)}{\sigma_\nu(\zeta)}_{(h^{-1})^t}\nu]$. In view of~(\ref{eq:Riemannian.dual-metric-g}) and~(\ref{eq:sigma-nu-forms}) we have
\begin{align*}
\acoup{\sigma_\nu(\omega)}{\sigma_\nu(\zeta)}_{(h^{-1})^t}\nu  & = \sum_{1\leq i,j \leq n} \big(\nu^{\frac12} \zeta_i\nu^{-\frac12}\big)^* h^{ij}  
\big(\nu^{\frac12} \omega_j  \nu^{-\frac12} \big) \nu\\ & = \sum_{1\leq i,j \leq n} \nu^{-\frac12}\big( \zeta_i^* \nu^{\frac12} h^{ij} \nu^{\frac12} \omega_j\big) \nu^{\frac12}. 
\end{align*}
As $\tau$ is a trace we obtain
\begin{equation*}
  \scal{\omega}{\zeta}_{h,\nu}^\opp =  \sum_{1\leq i,j \leq n} \tau  \big[\nu^{-\frac12}\big( \zeta_i^* \nu^{\frac12} h^{ij} \nu^{\frac12} \omega_j\big) \nu^{\frac12}\big]=  
  \sum_{1\leq i,j \leq n} \tau  \big[\zeta_i^*  \nu^{\frac12} h^{ij} \nu^{\frac12} \omega_j \big].  
\end{equation*}
The lemma is proved. 
\end{proof}

\begin{definition}
 $\cH^\opp_{h,\nu}(\Omega_\theta^1)$ is the completion of $\Omega^1_\theta$ with respect to the inner product $\scal{\omega}{\zeta}_{h,\nu}^\opp$. 
\end{definition}

\begin{remark}
 In the same way as with Lemma~\ref{lem:inner product-Hermitian} it can be shown that $\cH^\opp_{h,\nu}(\Omega_\theta^1)$ and its locally convex space structure do not depend on $\nu$ or $h$. 
\end{remark}

\begin{remark}
 When $g\in \GL^+_n(\cA_\theta^\R)$ we denote by $\scal{\cdot}{\cdot}_{g}^\opp$ the inner product $\scal{\cdot}{\cdot}_{g,\nu(g)}^\opp$. We shall also denote by $\cH_g(\Omega_\theta^1)$ the Hilbert space $\cH^\opp_{g,\nu(g)}(\Omega_\theta^1)$. 
\end{remark}

Let us now construct the divergence operator on $\Omega^1_\theta$. 

\begin{definition}
 The \emph{divergence} (with respect to $\nu$) of any $X=\sum X^{i} \dl_i\in \cX_\theta$ is 
 \begin{equation*}
 \dive_\nu(X) := \sum_{1\leq i \leq n} \dl_i \big( X^{i} \nu\big)\nu^{-1}. 
\end{equation*}
\end{definition}

\begin{proposition}
 The following holds. 
\begin{enumerate}
\item $\varphi_\nu( \dive_\nu(X))=0$ for all $X \in \cX_\theta$. 

\item Leibniz Rule: For all $u\in  \cA_\theta$ and $X= \sum_i X^{i} \dl_i \in \cX_\theta$, we have
\begin{equation*}
 \dive_\nu ( u X) = \nabla u \cdot X + u \dive_\nu(X), 
\end{equation*}
  where we have set $ \nabla u \cdot X= \sum_i \dl_i(u)X^{i}$. 
\end{enumerate}
\end{proposition}
\begin{proof}
 Let $X= \sum_i X^{i} \dl_i \in \cX_\theta $. By using~(\ref{eq:NCtori.closed-trace}) we get 
 \begin{equation*}
 \varphi_\nu( \dive_\nu(X))= \sum_{1\leq i \leq n} \tau\big[\dl_i\big( X^{i} \nu\big)\nu^{-1}\cdot \nu\big] = \sum_{1\leq i \leq n} \tau\big[\dl_i\big( X^{i} \nu\big)\big]=0. 
\end{equation*}
In addition, let $u\in \cA_\theta$. We have $\dive_\nu ( u X) = \sum_{i} \dl_i ( uX^{i} \nu)\nu^{-1}$. As $\dl_i ( uX^{i} \nu)\nu^{-1}= \dl_i(u) X^{i}+ u \dl_i(X^{i}\nu)\nu^{-1}$, we get
\begin{equation*}
 \dive_\nu ( u X) = \sum_{1\leq i \leq n} \dl_i(u) X^{i} +  \sum_{1\leq i \leq n}  u \dl_i(X^{i}\nu)\nu^{-1} = \nabla u \cdot X + u \dive_\nu(X).
\end{equation*}
The proof is complete. 
\end{proof}

\begin{definition}
 The divergence operator $\delta:\Omega_\theta^1 \rightarrow \cA_\theta$ is the linear operator defined by
 \begin{equation}
\delta\omega = \nu^{-1} \sum_{1\leq i,j \leq n} \dl_i \big(  \nu^{\frac12} h^{ij} \nu^{\frac12} \omega_j\big), \qquad \omega = \sum_{1\leq i \leq n} \theta^i \omega_i \in \Omega_\theta^1.
\label{eq:1-forms.div-op} 
\end{equation}
\end{definition}

The fact that $\delta$ is the analogue on $\Omega^1_\theta$ of the divergence operator on 1-forms on ordinary manifolds stems from the following result. 

\begin{proposition}\label{prop:1-forms.divergence}
The following holds. 
\begin{enumerate}
 \item The operator $-\delta$ is the formal adjoint of the differential $d:\cA_\theta \rightarrow \Omega_\theta^1$ with respect to the inner products $\scal{\cdot}{\cdot}_\nu^\opp$ and $\scal{\cdot}{\cdot}_{h,\nu}^\opp$. That is, we have 
 \begin{equation}
 \scal{-\delta\omega}{u}_\nu^\opp =  \scal{\omega}{du}_{h,\nu}^\opp \qquad \text{for all $\omega \in \Omega_\theta^1$ and $u\in \cA_\theta$}.
 \label{eq:1-forms.div-adjoint}
\end{equation}

 \item For all $\omega = \sum \theta^i \omega_i$ in $\Omega_\theta^1$, we have
\begin{equation*}
 \delta \omega = \big[ \dive_\nu(X_\omega^{h,\nu})\big]^*, \qquad \text{where $X^{h,\nu}_\omega= \sum_i \omega_j^*  \nu^{\frac12} h^{ji} \nu^{-\frac12}\dl_i$}. 
\end{equation*}
In particular, $ \delta \omega =[\dive_\nu (X_\omega^h)]^*$ when $[h,\nu]=0$.
\end{enumerate}
\end{proposition}
\begin{proof}
 Let $\omega = \sum \theta^i \omega_i\in \Omega_\theta^1$ and $u \in \cA_\theta$. As $du= \sum \theta^i \dl_i(u)$, by using~(\ref{eq:div.inner product-forms-nu}) we get
\begin{equation*}
 \scal{\omega}{du}_{h,\nu}^\opp =  \sum_{1\leq i,j \leq n} \tau\big[ (\dl_i u)^*\nu^{\frac12} h^{ij} \nu^{\frac12} \omega_j\big]. 
\end{equation*}
 It follows from~(\ref{eq:NCtori.derivation-involution}) and~(\ref{eq:NCtori.integration-by-parts}) that we have
\begin{equation*}
\tau\big[ (\dl_i u)^*\nu^{\frac12} h^{ij} \nu^{\frac12} \omega_j\big] =  
  \tau\big[ \dl_i (u^*) \nu^{\frac12} h^{ij} \nu^{\frac12}  \omega_j\big]= - \tau\left[ u^* \dl_i \big(\nu^{\frac12} h^{ij} \nu^{\frac12} \omega_j\big) \right].
\end{equation*}
As by definition $\delta \omega= \nu^{-1}\sum_{i,j} \dl_i (\nu^{\frac12} h^{ij} \nu^{\frac12} \omega_j)$, we  see that, for all $u \in \cA_\theta$, we have 
\begin{equation*}
 \scal{\omega}{du}_{h,\nu}^\opp = - \sum_{1\leq i,j \leq n}  \tau\left[ u^* \dl_i \big(\nu^{\frac12} h^{ij} \nu^{\frac12} \omega_j\big) \right] =- \tau\big[ u^* \nu (\delta \omega)\big]= \scal{-\delta \omega}{u}^\opp_\nu.  
\end{equation*}
This shows that $-\delta$ is the formal adjoint of $d$. Note that the unicity of the formal adjoint is a consequence of the density of $\cA_\theta$ in $\cH^\opp_\nu$. 

We also observe that $(\delta \omega)^*$ is equal to 
\begin{align*}
  \sum_{1\leq i,j \leq n} \left[\dl_i \big(  \nu^{\frac12} h^{ij} \nu^{\frac12} \omega_j\big)\right]^* \nu^{-1} 
& =  \sum_{1\leq i,j \leq n} \dl_i \big( \omega_j^* \nu^{\frac12} (h^{ij})^* \nu^{\frac12} \big) \nu^{-1}\\
&  =  \sum_{1\leq i,j \leq n} \dl_i \big( \omega_j^* \nu^{\frac12} h^{ji}
  \nu^{-\frac12}\cdot \nu \big) \nu^{-1}.
\end{align*}
Thus, if we set $X^{h,\nu}_\omega= \sum_{i,j} \omega_j^*  \nu^{\frac12} h^{ji} \nu^{-\frac12}\dl_i$, then 
 $(\delta \omega)^* =\dive_\nu (X^{h,\nu}_\omega)$, i.e.,  $\delta \omega = [ \dive_\nu(X_\omega^{h,\nu})]^*$. Furthermore, if $[\nu,h]=0$, then by using~(\ref{eq:Riemannian.Xomg}) we get $X^{h,\nu}_\omega= \sum_{i,j} \omega_j^*  h^{ij} \dl_i=X^h_\omega$, and hence $ \delta \omega =[\dive_\nu (X_\omega^h)]^*$. The proof is complete.  
\end{proof}

\section{Laplace-Beltrami Operators}\label{sec:Laplace-Beltrami}
In this section, we define and study the main properties of Laplace-Beltrami operators associated with arbitrary Hermitian metrics and densities on
$\cA_\theta$. 

\subsection{Definition and examples} In what follows we let $h\in \GL_n^+(\cA_\theta)$ and $\nu \in \cA_\theta^{++}$.  We also let $\delta: \Omega^1_\theta \rightarrow \cA_\theta$ be the divergence operator~(\ref{eq:1-forms.div-op}) associated with $(h,\nu)$. 

\begin{definition}
The Laplace-Beltrami operator $\Delta_{h,\nu}:\cA_\theta \rightarrow \cA_\theta$ is defined by
\begin{equation}
 \Delta_{h,\nu}(u)=- \delta (du), \qquad u \in \cA_\theta. 
 \label{eq:Laplacian.definition}
\end{equation}
\end{definition}
 
\begin{remark}
 When $g\in \GL_n^+(\cA_\theta^\R)$ we shall denote  by $\Delta_g$ the operator $\Delta_{g,\nu(g)}$. 
 This is the Laplace-Beltrami operator associated with the Riemannian metric $g$.  
\end{remark}

Recall that in the terminology of~\cite{Co:CRAS80} a \emph{differential operator} of order $m$ is any linear operator $P:\cA_\theta \rightarrow \cA_\theta$ of the form, 
\begin{equation*}
 P= \sum_{|\alpha|\leq m} a_\alpha \dl^\alpha, \qquad a_\alpha \in \cA_\theta. 
\end{equation*}
The \emph{symbol} of $P$ is $\rho(\xi) := \sum_{|\alpha|\leq m} a_\alpha i^{|\alpha|} \xi^\alpha$, $\xi\in \R^n$. This is a polynomial map on $\R^n$ with coefficients in $\cA_\theta$ of degree $m$. The degree~$m$ component $\rho_m(\xi):=\sum_{|\alpha|=m}a_\alpha i^m \xi^\alpha$ is called the \emph{principal symbol} of $P$. 

We say that $P$ is \emph{elliptic} when the principal symbol $\rho_m(\xi)$ is invertible for all $\xi\in \R^n\setminus 0$. In this case $\rho_m(\xi)^{-1}\in C^\infty(\R^n\setminus 0; \cA_\theta)$ and $P$ admits a parametrix in the class of Connes' pseudodifferential operators on $\cA_\theta$ (see~\cite{Ba:CRAS88, Co:CRAS80, HLP:Part1, HLP:Part2}). We refer to these references for general properties of elliptic operators on noncommutative tori. 

\begin{proposition}\label{prop:Laplacian.positivity-ellipticity}
 The following holds. 
\begin{enumerate}
 \item We have
         \begin{equation*}
               \scal{\Delta_{h,\nu}u}{v}_\nu^\opp = \scal{du}{dv}_{h,\nu}^\opp, \qquad u,v\in \cA_\theta. 
         \end{equation*}
 \item For all $u\in \cA_\theta$, we have 
 \begin{align}
 \Delta_{h,\nu}u & = - \nu^{-1} \sum_{1\leq i,j \leq n} \dl_i \big( \sqrt{\nu} h^{ij} \sqrt{\nu} \dl_j(u)\big) \nonumber\\
  &= -\sum_{1\leq i,j \leq n} \nu^{-\frac{1}{2}} h^{ij} \nu^{\frac12} \dl_i \dl_j(u) - \nu^{-1} \sum_{1\leq i,j \leq n} \dl_i \big( \sqrt{\nu} h^{ij} \sqrt{\nu}\big) \dl_j(u). 
  \label{eq:Laplacian.formula}
\end{align}    

\item $\Delta_{h,\nu}$ is an elliptic 2nd order differential operator with principal symbol $\rho_2(\xi)= \nu^{-\frac{1}{2}} (\xi,\xi)_{h^{-1}}\nu^{\frac12}$.    
\end{enumerate}
\end{proposition}
\begin{proof}
 Let $u,v\in \cA_\theta$. It follows from~(\ref{eq:1-forms.div-adjoint}) that
 \begin{equation}
 \scal{\Delta_{h,\nu}u}{v}_\nu^\opp = \scal{-\delta(du)}{v}_\nu^\opp = \scal{du}{dv}_{h,\nu}^\opp.
 \label{eq:Laplacian.positivity}
\end{equation}
Moreover, as $du = \sum \theta^i \dl_i(u)$, by using~(\ref{eq:1-forms.div-op}) we get 
\begin{equation*}
        \Delta_{h,\nu}u = -\delta (du)  = -\nu^{-1} \sum_{1\leq i,j \leq n} \dl_i \big( \sqrt{\nu} h^{ij} \sqrt{\nu} \dl_j(u) \big). 
\end{equation*}
As $ \dl_i ( \sqrt{\nu} h^{ij} \sqrt{\nu} \dl_j(u))=  \sqrt{\nu} h^{ij} \sqrt{\nu} \dl_i\dl_j(u) + \dl_i ( \sqrt{\nu} h^{ij} \sqrt{\nu}) \dl_j(u)$, we get
\begin{equation*}
  \Delta_{h,\nu}u =  -\sum_{1\leq i,j \leq n} \nu^{-\frac{1}{2}} h^{ij} \nu^{\frac12} \dl_i \dl_j(u) - \nu^{-1} \sum_{1\leq i,j \leq n} \dl_i \big( \sqrt{\nu} h^{ij} \sqrt{\nu}\big) \dl_j(u).
\end{equation*}
In particular, we see that $\Delta_{h,\nu}$ is a 2nd order differential operator with principal symbol, 
\begin{equation*}
   \rho_2(\xi)= \sum_{1\leq i,j \leq n} \nu^{-\frac{1}{2}} h^{ij} \nu^{\frac12} \xi_i \xi_j =  \nu^{-\frac{1}{2}} (\xi,\xi)_{h^{-1}}\nu^{\frac12}.  
\end{equation*}
Let $\xi \in \R^n\setminus 0$. As $h^{-1}\in \GL^+_n(\cA_\theta)$, we know by Corollary~\ref{cor:Positivity.equivalence-free} that  $(\xi,\xi)_{h^{-1}}$ is in $\cA_\theta^{++}$. In particular, it is invertible, and so $\rho_2(\xi)=\nu^{-\frac{1}{2}} (\xi,\xi)_{h^{-1}}\nu^{\frac12}$ is invertible as well. This shows that $\Delta_{h,\nu}$ is elliptic. The proof is complete. 
\end{proof}

\begin{remark}
 Let $g \in \GL_n^+(\cA_\theta^\R)$. When $\nu=\nu(g)$ the formulas~(\ref{eq:Laplacian.formula}) give
 \begin{equation}
 \Delta_g u=  - \nu(g)^{-1} \sum_{1\leq i,j \leq n} \dl_i \big( \sqrt{\nu(g)} g^{ij} \sqrt{\nu(g)} \dl_j(u)\big), \qquad u \in \cA_\theta.
 \label{eq:Laplacian.formula-vg}
\end{equation}
Assume further that $g$ is self-compatible. Its determinant belongs to the commutative unital $C^*$-algebra generated by its entries. Then $\sqrt{\nu(g)}$ belongs to this $C^*$-algebra as well, and so $[\sqrt{\nu(g)},g]=0$. This implies that $[\sqrt{\nu(g)},g^{-1}]=0$, i.e., $\sqrt{\nu(g)}$ commutes with the entries of $g^{-1}$. Thus, in this case we can rewrite~(\ref{eq:Laplacian.formula-vg}) in the form, 
\begin{equation}
  \Delta_g u =  \frac{-1}{\sqrt{\det (g)}} \sum_{1\leq i,j \leq n} \dl_i \big(  \sqrt{\det(g)} g^{ij} \dl_j(u)\big), \qquad u\in \cA_\theta.
  \label{eq:Laplacian.formula-g} 
\end{equation}
Here $\det(g)$ is given by~(\ref{eq:det.Leibniz}), since $g$ is self-compatible (\emph{cf}.~Corollary~\ref{cor:det.self-comp}). Therefore, we recover the standard formula for the Laplace-Beltrami operator associated with a Riemannian metric. 
\end{remark}

\begin{example}
 Let $g=\delta_{ij}$ be the Euclidean metric. Then $\det(g)=1$ and $g^{ij}=\delta^{ij}$. Thus, in this case~(\ref{eq:Laplacian.formula-g}) gives
\begin{equation*}
 \Delta_g  = - \sum_{1\leq i,j \leq n} \dl_i \delta^{ij} \dl_j=  - \sum_{1\leq i\leq n} \dl_i^2.
\end{equation*}
Therefore, we recover the usual Laplace operator $\Delta :=-\dl_1^2-\cdots - \dl_n^2$ on $\cA_\theta$. 
\end{example}

\begin{example}\label{ex:Laplacian.conformal-flat}
 Let $g_{ij}=k^{2} \delta_{ij}$, $k \in \cA_\theta^{++}$, be a conformal deformation of the Euclidean metric.Then $g^{ij}=k^{-2}\delta^{ij}$ and we know by Proposition~\ref{prop:det.properties} that $\det (g)=k^{2n}$. Thus, by using~(\ref{eq:Laplacian.formula-g}) we get
\begin{equation*}
 \Delta_g u = - k^{-n}  \sum_{1\leq i,j \leq n} \dl_j \big(k^{n} k^{-2}\delta^{ij} \dl_j(u)\big)= - k^{-n} \sum_{1\leq i \leq n} \dl_i \big(k^{n-2} \dl_i(u)\big). 
\end{equation*}
As $ \dl_i (k^{n-2} \dl_i(u))= k^{n-2}\dl_i^2(u) + \dl_i(k^{n-2})\dl_i(u)$, we obtain 
\begin{equation*}
  \Delta_g u = - k^{-2} \sum_{1\leq i \leq n} \dl_i^2(u)- \sum_{1\leq i \leq n} k^{-n}  \dl_i(k^{n-2}) \dl_i(u) = k^{-2} \Delta u- \sum_{1\leq i \leq n} k^{-n}  \dl_i(k^{n-2}) \dl_i(u).
\end{equation*}
Suppose that $n=2$. In this case $k^{n-2}=1$, and so $\dl_i(k^{n-2})=0$ for $i=1,\ldots, n$. Thus, we get
\begin{equation*}
 \Delta_g = k^{-2} \Delta. 
\end{equation*}
Here $\nu(g)=k^2$, and so the left-multiplication by $k$ gives rise to a unitary isomorphism from $\cH_{g}^\opp$ onto $\cH_\theta$. Therefore, $ \Delta_g $ is unitary equivalent to the operator, 
\begin{equation*}
 k \Delta_g k^{-1}= k^{-1} \Delta k^{-1}.  
\end{equation*}
Thus, up to unitary equivalence, we recover the conformally deformed Laplacian $k^{-1} \Delta k^{-1}$ of Connes-Tretkoff~\cite{CT:Baltimore11}.  This operator was also considered in several subsequent papers (see, e.g., \cite{CM:JAMS14, DGK:arXiv18, Fa:JMP15, FK:JNCG12, FK:LMP13,  FK:JNCG13, FK:JNCG15, FGK:JNCG19, LM:GAFA16, Liu:arXiv18a, Liu:arXiv18b}). 
\end{example}

\begin{example}
 Suppose that $n=2$ and $g$ is a product metric of the form, 
 \begin{equation*}
 g= \begin{pmatrix}
 (k_1)^2 & 0\\
 0 & (k_2)^2
\end{pmatrix}, \qquad k_j \in \cA_\theta^{++}. 
\end{equation*}
We have 
\begin{equation*}
 g^{-1}= \begin{pmatrix}
 k_1^{-2} & 0\\
 0 & k_2^{-2}
\end{pmatrix}. 
\end{equation*}
Assume further that $[k_1,k_2]=0$. In this case $\det(g)=(k_1)^2(k_2)^2=(k_1k_2)^2$. Thus, by using~(\ref{eq:Laplacian.formula-g}) we get
\begin{align*}
- \Delta_g u & = \frac{1}{k_1 k_2} \sum_{1\leq i,j\leq 2} \dl_i \big( k_1k_2 g^{ij} \dl_j(u)\big) \\
& = \frac{1}{k_1 k_2} \left( \dl_1 \big[ k_2k_1^{-1}\dl_1(u) \big] + \dl_2 \big[ k_1k_2^{-1}\dl_2(u) \big]\right)\\ 
& = k_1^{-2} \dl_1^2(u) + k_2^{-2} \dl_2^2(u) + \frac{1}{k_1 k_2} \left( \dl_1 \big[ k_2k_1^{-1}\big]\dl_1(u) +  \dl_2 \big[ k_1k_2^{-1}\big]\dl_2(u) \right). 
\end{align*}
\end{example}

\begin{proposition}\label{prop:Laplacian.conformal-transformation} 
 Let $g\in \GL_n^+(\cA_\theta^\R)$ and $k\in \cA_\theta^{++}$ be such that $[k,g]=0$. Set $\hat{g}=k gk=k^2g$. Then 
 \begin{equation*}
\Delta_{\hat{g}}u = k^{-2} \Delta_gu - \nu(g)^{-1}k^{-n} \acoup{\sqrt{\nu(g)} du}{\sqrt{\nu(g)} d(k^{n-2})}_{g^{-1}}, \qquad u\in \cA_\theta.
\end{equation*}
In particular, when $n=2$ we have 
\begin{equation}
 \Delta_{\hat{g}} = k^{-2} \Delta_g.
 \label{eq:Laplacian.conformal-transformation-2D} 
\end{equation}
\end{proposition}
\begin{proof}
 Let $u\in \cA_\theta$. As $[k,g]=0$ we have $\hat{g}^{-1}=k^{-2}g^{-1}$, i.e., $\hat{g}^{ij}=k^{-2}g^{ij}$. We also know by Proposition~\ref{prop:volume.conformal-transformation} that $\nu(\hat{g})=k^{n}\nu(g)=\nu(g)k^{n}$, and so $\sqrt{\nu(g)}= k^{\frac{n}2} \sqrt{\nu(g)}=\sqrt{\nu(g)}k^{\frac{n}2}$. Combining this with~(\ref{eq:Laplacian.formula-vg}) then gives
\begin{align*}
 \Delta_{\hat{g}}u =  & -k^{-n}\nu(g)^{-1} \sum_{1\leq i,j\leq n} \dl_i\big[ \sqrt{\nu(g)} k^{\frac{n}2} k^{-2} g^{ij} k^{\frac{n}2} \sqrt{\nu(g)}\dl_j(u)\big]\\
 = &  -k^{-n}\nu(g)^{-1} \sum_{1\leq i,j\leq n} \dl_i\big[k^{n-2} \sqrt{\nu(g)}  g^{ij} \sqrt{\nu(g)}\dl_j(u)\big]\\
 = & -k^{-2}\nu(g)^{-1} \sum_{1\leq i,j\leq n} \dl_i\big[\sqrt{\nu(g)}  g^{ij} \sqrt{\nu(g)}\dl_j(u)\big]\\
 & - k^{-n}\nu(g)^{-1} \sum_{1\leq i,j\leq n} \dl_i\big(k^{n-2}\big) \sqrt{\nu(g)}  g^{ij} \sqrt{\nu(g)}\dl_j(u)\\
 = & k^{-2} \Delta_g(u) - \nu(g)^{-1}k^{-n} \acoup{\sqrt{\nu(g)} du}{\sqrt{\nu(g)} d(k^{n-2})}_{g^{-1}}. 
\end{align*}
When $n=2$, in the same way as with Example~\ref{ex:Laplacian.conformal-flat}, we have $d(k^{n-2})=d(1)=0$, and so we simply get $ \Delta_{\hat{g}}u=k^{-2} \Delta_gu$. This completes the proof. 
\end{proof}

\begin{remark}
 Example~\ref{ex:Laplacian.conformal-flat} is a special case of Proposition~\ref{prop:Laplacian.conformal-transformation} with $g_{ij}=\delta_{ij}$.
 \end{remark}

\begin{remark}
 The transformation law~(\ref{eq:Laplacian.conformal-transformation-2D}) is the version for noncommutative 2-tori of the conformal covariance of the Laplace-Beltrami operator on Riemann surfaces. 
\end{remark}
 
 \subsection{Elliptic regularity properties} 
 Let $\cA_\theta'$ be the topological dual of $\cA_\theta$. We equip it with its strong dual topology. We think of $\cA_\theta'$ as the space of distributions on $\cA_\theta$. We note there is a natural inclusion of $\cA_\theta$ into $\cA_\theta'$ induced by the natural pairing, 
\begin{equation}
 \acou{v}{u}:= \tau\big[ vu\big], \qquad u,v\in \cA_\theta.
 \label{eq:Laplacian.pairingAA}
\end{equation}
 More precisely, $\acou{v}{\cdot}\in \cA_\theta'$ and $v \rightarrow \acou{v}{\cdot}$ is a continuous embedding from $\cA_\theta$ into $\cA_\theta'$. This allows us to identify $\cA_\theta$ with its image in $\cA_\theta'$ under this embedding. Note that every $v\in \cA_\theta'$ is sum of its Fourier series in $\cA_\theta'$. Namely, 
\begin{equation}
 v = \sum_{k\in \Z^n} v_k U^k, \qquad v_k:=\acou{v}{(U^k)^*},
 \label{eq:Laplacian.Fourier-series-A'}
\end{equation}
 where the series converges in $\cA_\theta'$  (see, e.g., \cite{HLP:Part2, XXY:MAMS18}). In particular, this implies that $\cA_\theta$ is dense in $\cA_\theta'$. In addition, every differential operator $P:\cA_\theta \rightarrow \cA_\theta$ uniquely extends to a continuous linear operator $P:\cA_\theta'\rightarrow \cA_\theta'$ (see, e.g., \cite{HLP:Part2}). In particular, for any $\alpha \in \N_0^n$, we have 
 \begin{equation*}
 \acou{\dl^\alpha v}{u}=(-1)^{|\alpha|} \acou{v}{\dl^\alpha u}, \qquad v\in \cA_\theta', \ u \in \cA_\theta. 
\end{equation*}

Given any $s\in \R$ the Sobolev space $\cH_\theta^{(s)}$ is defined by 
\begin{equation*}
 \cH^{(s)}_\theta = \biggl\{v \in \cA_\theta'; \ \sum_{k\in \Z^n} (1+|k|^2)^{s}|v_k|^2<\infty \biggr\}.  
\end{equation*}
This is a Hilbert space with respect to the Hilbert norm, 
\begin{equation*}
 \|v\|_{s}= \biggl(\sum_{k\in \Z^n} (1+|k|^2)^{s}|v_k|^2\biggr)^{\frac12}, \qquad v\in \cH_\theta^{(s)}. 
\end{equation*}
For $s=0$ we recover the Hilbert space $\cH_\theta$. More generally, given any $m\in \N$, we have 
\begin{equation*}
 \cH^{(m)}_\theta = \bigg\{u \in \cH_\theta; \ \dl^\alpha u \in \cH_\theta \ \text{for $|\alpha|\leq m$} \biggr\}. 
\end{equation*}

Let $s\in \R$. If $v\in \cH_\theta^{(s)}$, then its Fourier series~(\ref{eq:Laplacian.Fourier-series-A'}) actually converges in $\cH_\theta^{(s)}$. This shows that $\cA_\theta$ is a dense subspace of $\cH_\theta^{(s)}$. In addition, the natural pairing~(\ref{eq:Laplacian.pairingAA}) on $\cA_\theta\times \cA_\theta$ uniquely extends to a continuous bilinear pairing on $\cH_\theta^{(-s)}\times \cH_\theta^{(s)}$. This pairing is non-degenerate, and so we have a canonical isomorphism $(\cH_\theta^{(s)})'\simeq \cH_\theta^{(-s)}$. 

We also have an analogue of Sobolev's embedding theorem. If $t>s$, then  the inclusion $\cH^{(t)}_\theta \subset \cH_\theta^{(s)}$ is compact. Moreover, we have  
\begin{equation*}
 \cA_\theta = \bigcap_{s\in \R} \cH_\theta^{(s)} \qquad \text{and} \qquad \bigcup_{s\in \R} \cH_\theta^{(s)}= \cA_\theta'. 
\end{equation*}
In fact, at the topological level, the Sobolev norms $\|\cdot\|_s$, $s\in \R$, generate the topology of $\cA_\theta$. Moreover, the strong topology of $\cA_\theta'$ is the direct limit of the  $\|\cdot\|_s$ topologies. 

We refer to~\cite{HLP:Part2, Sp:Padova92, XXY:MAMS18} for more complete accounts on Sobolev spaces on noncommutative tori, including proofs for the results mentioned above. 

Let $h\in \GL_m^+(\cA_\theta)$ and $\nu \in \cA_\theta^{++}$. As the Laplace-Beltrami operator $\Delta_{h,\nu}$ is a differential operator, it uniquely extends to a continuous linear operator $\Delta_{h,\nu}:\cA_\theta'\rightarrow \cA_\theta'$. In fact,  by Proposition~\ref{prop:Laplacian.positivity-ellipticity} this is an \emph{elliptic} 2nd order operator, and so the elliptic regularity results for such operators in~\cite{HLP:Part2} apply. We summarize them in the following statement. 

\begin{proposition}\label{prop:Laplacian.elliptic-regularity}
 The following holds. 
\begin{enumerate}
 \item For every $s\in \R$, the operator $\Delta_{h,\nu}$ uniquely extends to a continuous linear operator $\Delta_{h,\nu}:\cH_\theta^{(s+2)}\rightarrow \cH_\theta^{(s)}$. 
 
 \item Let $s\in \R$. For any $u\in \cA_\theta'$, we have
\begin{equation}
 \Delta_{h,\nu}u \in \cH_\theta^{(s)} \Longleftrightarrow u \in \cH_\theta^{(s+2)}.
 \label{eq:Laplacian.elliptic-regularity-Sobolev}  
\end{equation}

\item The operator $\Delta_{h,\nu}$ is hypoelliptic, i.e., for any $u \in \cA_\theta'$, we have 
\begin{equation*}
    \Delta_{h,\nu}u \in \cA_\theta  \Longleftrightarrow u \in \cA_\theta.
\end{equation*}
\end{enumerate}
\end{proposition}

\subsection{Spectral theory} Let us now turn to the spectral theory of $\Delta_{h,\nu}$. As $\cH_\nu^\opp$ is the same locally convex space as $\cH_\theta=\cH_\theta^{(0)}$, the pairing~(\ref{eq:Laplacian.pairingAA}) gives rise to a non-degenerate pairing on $\cH_\nu^\opp$. This pairing is related to the inner product $\scal{\cdot}{\cdot}_\nu^\opp$ as follows. Given any $u,v\in \cA_\theta$, we have 
\begin{equation*}
 \acou{u}{v}= \tau[vu]= \tau\big[ (\nu^{-1}v^*)^*\nu u\big] = \scal{u}{\nu^{-1}v^*}_\nu^\opp. 
\end{equation*}
Combining this with the density of $\cA_\theta$ in $\cH_\nu^\opp$, we deduce that
\begin{equation}
 \acou{u}{v}=  \scal{u}{\nu^{-1}v^*}_\nu^\opp \qquad \text{for all $u,v\in \cH_\nu^\opp$}.
 \label{eq:Laplacian.pairing-inner-nu-opp} 
\end{equation}

Bearing this in mind, by Proposition \ref{prop:Laplacian.elliptic-regularity} the operator $\Delta_{h,\nu}$ maps continuously $\cH_\theta^{(2)}$ to $\cH_\nu^\opp$ and $\cH_\theta^{(2)}$ is a maximal domain for $\Delta_{h,\nu}$ on $\cH_\nu^\opp$. We actually have the following result. 

\begin{proposition}\label{prop:Laplacian.spectral-theory}
 The following holds. 
\begin{enumerate}
\item $\Delta_{h,\nu}$ is selfadjoint  when we regard it as an operator on $\cH_\nu^\opp$ with domain $\cH_\theta^{(2)}$. 

\item The spectrum of $\Delta_{h,\nu}$ is an unbounded discrete subset of $[0,\infty)$ that consists of eigenvalues with finite multiplicity. 

\item Every eigenspace of $\Delta_{h,\nu}$ is a finite dimensional subspace of $\cA_\theta$. 
\end{enumerate}
\end{proposition}
\begin{proof}
 Let us first show that $\Delta_{h,\nu}$ is selfadjoint. We regard it as an operator on $\cH_\nu^\opp$ with domain $\cH_\theta^{(2)}$. It follows from~(\ref{eq:Laplacian.positivity}) that $\Delta_{h,\nu}$ is formally selfadjoint with respect to the inner product $\scal{\cdot}{\cdot}_\nu^\opp$. Namely, for all $u,v\in \cA_\theta$, we have 
\begin{equation}
 \scal{\Delta_{h,\nu} u}{v}_\nu^\opp=  \scal{du}{dv}_{h,\nu}^\opp=\overline{\scal{dv}{du}_{h,\nu}^\opp}=\overline{\scal{\Delta_{h,\nu} v}{u}_\nu^\opp}=
 \scal{u}{\Delta_{h,\nu} v}_\nu^\opp. 
 \label{eq:Laplacian.selfadjoint}
\end{equation}
Combining this with the density of $\cA_\theta$ in $\cH_\theta^{(2)}$ and the fact that $\Delta_{h,\nu}$ maps continuously $\cH_\theta^{(2)}$ to $\cH_\nu^\opp$, we deduce that
\begin{equation*}
 \scal{\Delta_{h,\nu} u}{v}_\nu^\opp=\scal{u}{\Delta_{h,\nu} v}_\nu^\opp \qquad \text{for all $u,v\in \cH_\theta^{(2)}$}. 
\end{equation*}
That is, $\Delta_{h,\nu}$ is symmetric, i.e., $\Delta_{h,\nu} \subset \Delta_{h,\nu}^*$. 

Let $u \in \dom(\Delta_{h,\nu}^*)$. In addition, let $(u_\ell)_{\ell\geq 0}\subset \cA_\theta$ be a sequence such that $u_\ell \rightarrow u$ in $\cH_\nu^\opp$. It follows from Proposition~\ref{prop:Laplacian.elliptic-regularity} that $\Delta_{h,\nu}$ maps continuously $\cH_\nu^\opp=\cH_\theta^{(0)}$ to $\cH_\theta^{(-2)}$. Thus, $\Delta_{h,\nu}u_\ell \rightarrow \Delta_{h,\nu}u$ in $\cH_\theta^{(-2)}$, and hence $\Delta_{h,\nu}u_\ell \rightarrow \Delta_{h,\nu}u$ in $\cA_\theta'$. 

Let $v\in \cA_\theta'$, and set $\tilde{v}=\nu^{-1}v^*$. Then by~(\ref{eq:Laplacian.pairing-inner-nu-opp}) and~(\ref{eq:Laplacian.selfadjoint}) we have
\begin{equation*}
 \acou{\Delta_{h,\nu}u_\ell}{v}= \scal{\Delta_{h,\nu}u_\ell}{\tilde{v}}_\nu^\opp = \scal{u_\ell}{\Delta_{h,\nu}\tilde{v}}_\nu^\opp \longrightarrow  
 \scal{u}{\Delta_{h,\nu}\tilde{v}}_\nu^\opp. 
\end{equation*}
 Note that as $u \in \dom(\Delta_{h,\nu}^*)$, we have $\scalt{u}{\Delta_{h,\nu}\tilde{v}}_\nu^\opp=  \scalt{\Delta_{h,\nu}^*u}{\tilde{v}}_\nu^\opp=\acout{\Delta_{h,\nu}^*u}{v}$. Therefore, we see that $\acout{\Delta_{h,\nu}u_\ell}{v}\rightarrow \acout{\Delta_{h,\nu}^*u}{v}$ for all $v\in \cA_\theta$, i.e., $\Delta_{h,\nu}u_\ell \rightarrow \Delta_{h,\nu}^*u$ weakly in $\cA_\theta'$. As  we also know that $\Delta_{h,\nu}u_\ell \rightarrow \Delta_{h,\nu}u$  strongly in $\cA_\theta'$, we deduce that $\Delta_{h,\nu}u= \Delta_{h,\nu}^*u$. In particular, $\Delta_{h,\nu}u$ is contained in $\cH_\nu^\opp=\cH_\theta^{(0)}$, and so by~(\ref{eq:Laplacian.elliptic-regularity-Sobolev})  $u$ must be contained in $\cH_\theta^{(2)}$. 
 This shows that $\dom(\Delta_{h,\nu}^*)\subset \cH_\theta^{(2)}=\dom(\Delta_{h,\nu})$. As $\Delta_{h,\nu} \subset \Delta_{h,\nu}^*$, it then follows that $\Delta_{h,\nu} =\Delta_{h,\nu}^*$, i.e., $\Delta_{h,\nu}$ is selfadjoint. 
 
 Recall that $\Delta_{h,\nu}$ is an elliptic differential operator. Moreover,  the very fact that  $\Delta_{h,\nu}$ is selfadjoint implies that its spectrum is real. In particular, $\Sp( \Delta_{h,\nu})\neq \C$. It then follows from the results of~\cite{HLP:Part2} that $\Sp( \Delta_{h,\nu})$ is an unbounded discrete set that consists of eigenvalues with finite multiplicity. Moreover, all its eigenspaces are contained in $\cA_\theta$. We further observe that (\ref{eq:Laplacian.positivity}) implies that $\scal{\Delta_{h,\nu}u}{u}_\nu^\opp\geq 0$ for all $u\in \cA_\theta$. Thus, $\Delta_{h,\nu}$ cannot have negative eigenvalues, and so its spectrum must be contained in $[0,\infty)$. The proof is complete. 
\end{proof}

\begin{proposition}\label{prop:Laplacian.nullspace}
 We have
\begin{equation*}
 \ker \Delta_{h,\nu} = \ker d =\C. 
\end{equation*}
\end{proposition}
\begin{proof}
 We know by Proposition~\ref{prop:1-forms.properties-differential} and Proposition~\ref{prop:Laplacian.spectral-theory} that $\ker d =\C$ and $\ker \Delta_{h,\nu}\subset \cA_\theta$. Moreover, if $du=0$, then by~(\ref{eq:Laplacian.definition}) we have $\Delta_{h,\nu}u = -\delta(du)=0$. Thus, $\ker d \subset \ker \Delta_{h,\nu}$. In addition,  if $u \in \ker \Delta_{h,\nu}$, then $u\in \cA_\theta $, and by~(\ref{eq:Laplacian.positivity}) we have 
\begin{equation*}
 \scal{du}{du}_{h,\nu}^\opp= \scal{\Delta_{h,\nu}u}{u}^\opp=0.
\end{equation*}
As $ \scal{\cdot}{\cdot}_{h,\nu}^\opp$ is an inner product, it then follows that $du=0$. Thus,  $\ker \Delta_{h,\nu}\subset \ker d$. It then follows that $\ker \Delta_{h,\nu}=\ker d=\C$. The proof is complete.    
\end{proof}

By Proposition~\ref{prop:Laplacian.spectral-theory} the spectrum of $\Delta_{h,\nu}$ is an unbounded discrete set consisting of non-negative eigenvalues with finite multiplicities. Moreover, by Proposition~\ref{prop:Laplacian.nullspace} the zero-eigenvalue has multiplicity 1. Therefore, we can organize the spectrum of $\Delta_{h,\nu}$ as a non-decreasing sequence, 
\begin{equation*}
 0=\lambda_0( \Delta_{h,\nu}) < \lambda_1( \Delta_{h,\nu})\leq \lambda_2( \Delta_{h,\nu})\leq \cdots \leq \lambda_j(\Delta_{h,\nu}) \leq \cdots, 
\end{equation*}
where each eigenvalue is repeated according to multiplicity. The unboundedness of $\Sp( \Delta_{h,\nu})$ implies that $\lambda_\ell( \Delta_{h,\nu})\rightarrow \infty$ as $\ell \rightarrow \infty$. Moreover, as $\Delta_{h,\nu}$ is selfadjoint and all its eigenspaces are contained in $\cA_\theta$ we obtain the following result. 

\begin{proposition}\label{prop:Laplacian.eigenbasis}
 There is an orthonormal basis $(e_\ell)_{\ell\geq 0}$ of $\cH_\nu^\opp$ such that, for all $\ell\geq 0$, we have
\begin{equation*}
 \Delta_{h,\nu} e_\ell =  \lambda_\ell( \Delta_{h,\nu})e_\ell \qquad \text{and} \qquad e_\ell \in \cA_\theta. 
\end{equation*}
\end{proposition}

\begin{remark}[\cite{Po:Weyl}]\label{rmk:Laplacian.Weyl} 
 We have a Weyl law for the distribution of the eigenvalues $ \lambda_\ell( \Delta_{h,\nu})$. Namely, 
\begin{equation}
  \lambda_\ell( \Delta_{h,\nu}) \sim  \left( \frac{\ell}{c_n(h)}\right)^{\frac2{n}} \qquad \text{as $\ell \rightarrow \infty$}, 
  \label{eq:Laplacian.Weyl_Law}
\end{equation}
where we have set
\begin{equation*}
 c_n(h) = \frac{1}{n} \int_{\mathbb{S}^{n-1}} \tau\left[ |\xi|_{h}^{-n}\right]d^{n-1}\xi, \qquad |\xi|_h : =\sqrt{(\xi,\xi)_{h^{-1}}}= \sqrt{\sum \xi_i h^{ij}\xi_j}. 
\end{equation*}
In fact, when $g$ is a self-compatible Riemannian metric, we have 
\begin{equation*}
 c_n(g)= \frac{1}{n} (2\pi)^{-n}|\mathbb{S}^{n-1}| \Vol_g(\cA_\theta)= (2\pi)^{-n} |\mathbb{B}^n| \Vol_g(\cA_\theta). 
\end{equation*}
Thus, in this case we have the usual form of Weyl's law. In particular, when $n=2$ and $g$ is a conformal deformation of the Euclidean metric this agrees with the Weyl Law of~\cite{FK:LMP13}. We refer to~\cite{Po:Weyl} for a more detailed account, including local and microlocal versions of the Weyl 
law~(\ref{eq:Laplacian.Weyl_Law}).  
\end{remark}

\end{document}